\documentclass[11pt]{amsart}
\usepackage[latin1]{inputenc}
\usepackage{amsmath}
\usepackage{amsfonts}
\usepackage{amssymb,amsbsy,amsthm,mathtools}
\usepackage{color}
\usepackage[T1]{fontenc}
\usepackage[enableskew,vcentermath]{youngtab}
\usepackage{tikz}
\usepackage{comment}

\newcommand{\C}{C}

\DeclareMathOperator\sign{sign}
\newcommand{\Irr}{{\rm Irr}}

\DeclareMathOperator\lcm{lcm}
\DeclareMathOperator\id{id}

\def\CC{{\mathbb C}}

\def\NN{{\mathbb N}}

\def\ZZ{{\mathbb Z}}

\newcommand{\blambda}{{\underline\lambda}}
\newcommand{\bmu}{{\underline\mu}}
\newcommand{\br}{{\underline r}}
\newcommand{\bs}{{\underline s}}
\newcommand{\bt}{{\underline t}}
\newcommand{\bc}{{\underline c}}

\newcommand{\ve}{\varepsilon}

\newcommand{\ch}{{\operatorname{ch}}}

\newcommand{\roots}{\rho}
\newcommand{\sroots}{{\overline{\rho}}}
\newcommand{\spsi}{{\overline{\psi}}} 
\newcommand{\sk}{{k,-}}

\DeclareMathOperator\Conj{Conj}

\theoremstyle{plain}
\newtheorem{theorem}{Theorem}[section]
\newtheorem{proposition}[theorem]{Proposition}
\newtheorem{lemma}[theorem]{Lemma}
\newtheorem{corollary}[theorem]{Corollary}
\newtheorem{conjecture}[theorem]{Conjecture}
\newtheorem{problem}[theorem]{Problem}
\theoremstyle{definition}
\newtheorem{definition}[theorem]{Definition}

\newtheorem{question}[theorem]{Question}
\newtheorem{remark}[theorem]{Remark}
\newtheorem{observation}[theorem]{Observation}

\newcommand{\todo}[1]{\vspace{5 mm}\par \noindent
	\marginpar{\textsc{ToDo}} \framebox{\begin{minipage}[c]{0.9
				\textwidth}
			#1 \end{minipage}}\vspace{5 mm}\par}

\usepackage[english]{babel}

\usepackage[margin=1in,letterpaper,portrait]{geometry}

\usepackage{amsmath}
\usepackage{graphicx}
\usepackage[colorlinks=true, allcolors=blue]{hyperref}

\title[Higher Lie characters and root enumeration]{Higher Lie characters and root enumeration\\ in classical Weyl groups}

\author{Ron M.\ Adin}
\address{Department of Mathematics, Bar-Ilan University, 
Ramat-Gan 52900, Israel}
\email{radin@math.biu.ac.il}

\author{P\'al Heged\"us}
\address{Department of Algebra and Geometry, Institute of Mathematics, Budapest University of Technology and Economics, M\H uegyetem rkp. 3., H-1111 Budapest, Hungary}
\email{hegpal@math.bme.hu}

\author{Yuval Roichman}
\address{Department of Mathematics, Bar-Ilan University, 
Ramat-Gan 52900, Israel}
\email{yuvalr@math.biu.ac.il}

\date{October 30, 2024}

\thanks{RMA and YR were partially supported by The Israel Science Foundation, grant no.~1970/18. 
PH was partially supported by National Research, Development and Innovation Office - NKFIH (K138596).
}

\begin{document}
\maketitle

\begin{abstract} 
	We prove that, for any integer $k$, the $k$-th root enumerator in the classical Weyl group of type $D$ is a proper character.
	The proof uses higher Lie characters of type $B$. 
\end{abstract}

\tableofcontents

\section{Introduction}
\label{sec:introduction}

\subsection{Properness of root enumerators}
\label{subsec:intro_proper}

Let $G$ be a finite group. 
For every integer $k$ define the {\em $k$-th root enumerator} $\roots_k^G : G \to \ZZ$ by
\[
\roots_k^G(g) := |\{h \in G \,:\, h^k = g\}|.
\]
Clearly, $\roots_k^G$ is a class function on $G$. 
A classical result of Frobenius~\cite{Fr} (see also~\cite[Problem 4.7]{Isaacs}) shows that,  
for any finite group $G$, all root enumerators are virtual characters, 
namely: 
\[
\langle  \roots_k^G, \chi \rangle \in \ZZ \qquad (\forall \chi\in {\rm Irr}(G),\, k \in \ZZ),
\]
where ${\rm Irr}(G)$ is the set of all irreducible characters of $G$. 


The following problem is fundamental; see, e.g., \cite[pp.~101--102]{Scharf_thesis}. 


\begin{problem} 
Which finite groups $G$ have the property that all $k$-th root  enumerators $\roots^G_k$ are proper $G$-characters? 
Equivalently, for which finite groups $G$ is $\langle \roots_k^G, \chi \rangle \ge 0$ for all $\chi \in {\rm Irr}(G)$ and $k \in \ZZ$?
\end{problem}

Clearly, $\roots^G_0$ (the regular character) and $\roots^G_1$ (the trivial character) are proper characters for any finite group $G$.  
A complete characterization of the finite groups for which $\roots^G_2$ is a proper character 
was given in a celebrated work of Frobenius and Schur~\cite{FS}; see also~\cite[\S 4]{Isaacs}. 
In particular, if all the (complex) irreducible characters of $G$ can be obtained from real representations, then $\roots^G_2$ is a multiplicity-free sum of all the irreducible characters.

This paper deals with the classical Weyl groups;
for their classification into families ($A$, $B$ and $D$) see, e.g., \cite[\S 2.4]{Humphreys}. 
Kerber conjectured that, for the classical Weyl group of type $A_{n-1}$ (i.e., the symmetric group $S_n$ on $n$ letters), the $k$-th root enumerator is proper for all $n$ and $k$. 
This conjecture was proved by 
Scharf~\cite{Scharf}, presenting 
$\roots_k^{S_n}$ as a multiplicity-free sum of higher Lie characters (to be defined below).  
For a different (non-constructive) proof, see~\cite[Ex.~7.69.c]{ECII}. 
Scharf further proved~\cite[2.4.18]{Scharf_thesis} 
that, for any finite group $G$, the $k$-th root enumerator $\roots_k^{G \wr S_n}$ in the wreath product $G \wr S_n$ is a proper character, provided that $\roots^G_t$ is a proper character for all the divisors $t$ of $k$. 
In particular, all $k$-th root enumerators in the classical Weyl group of type $B_n$, isomorphic to $\ZZ_2 \wr S_n$, are proper. 
%
%
The problem remained open, for general $k$, for the third family of classical Weyl groups, those of type $D_n$.

Our main result is the following. 

\begin{theorem}\label{t:Dn-proper}
For any integer $k$ and positive integer $n$,  
the $k$-th root enumerator in the classical Weyl group of type $D_n$ is a proper character.    
\end{theorem}

Moreover, the root enumerators can be expressed as sums of higher Lie characters:     
Each root enumerator in $D_n$ is explicitly expressed as half the sum of restrictions (from $B_n$) of appropriate higher Lie characters of type $B$ (Theorem~\ref{t:Dn-higher}). 
We also obtain an explicit expression of each root enumerator in $B_n$ as a sum of higher Lie characters of type $B$ (Theorem~\ref{t:Bn-higher}). This gives a new proof of Scharf's result that root enumerators in $B_n$ are proper.

Combining Scharf's results with Theorem~\ref{t:Dn-proper}, we conclude the following.

\begin{corollary}\label{t:Weyl-proper}
For any integer $k$,  
the $k$-th root enumerator in any classical Weyl group  is a proper character.  
\end{corollary}

\subsection{Higher Lie characters and their role} 

The {\em Lie character} in $S_n$ is induced from a primitive linear character of a cyclic subgroup of order $n$, the centralizer $Z_{S_n}(x)$ of a long cycle (Coxeter element) $x \in S_n$.
It has many combinatorial, algebraic and topological applications; see, e.g.,~\cite{Garsia, Reutenauer}. A generalization to centralizers of other elements (in $S_n$) can be traced back to Schur~\cite{Schur} and Thrall~\cite{thrall}. These $S_n$-characters are now called {\em higher Lie characters}, and play an important role in the study of Lie algebras and 
permutation statistics; the reader is referred to the seminal paper~\cite{GR}, references therein and many subsequent works.  
Applications to root enumeration~\cite{Scharf} and Gelfand models~\cite{Inglis} motivate the current paper. 

Let us introduce the following natural definition, generalizing the $S_n$ scenario.

\begin{definition}\label{def:hLc}
{\rm (Higher Lie characters)}
Let $G$ be a finite group. 
A family of {\em higher Lie characters (HLC)} for $G$ is a family $\{\psi_G^C \,:\, C \in \Conj(G)\}$ of characters of $G$, indexed by conjugacy classes, 
satisfying the following conditions. 
\begin{enumerate}
	
	\item[(a)]
	For each $C \in \Conj(G)$ there is an element $g \in C$ such that $\psi_G^C$ is induced from a linear character of the centralizer $Z_G(g)$. 
	
	
	
	\item[(b)]
	For any integer $k$, the $k$-th root enumerator
	\[
	\roots_k^G = 
	\sum_{C \in \Conj(G) \,:\, C^k = \{1_G\}} \psi_G^C ,
	\]
	where $C^k := \{x^k \,:\, x \in C\}$.
	
\end{enumerate}
\end{definition}

Such characters were constructed for conjugacy classes of involutions $(k = 2)$ in $B_n$ by Baddeley~\cite{Baddeley} and by Scharf~\cite[\S 3.1]{Scharf_thesis}. The construction was generalized by Scharf~\cite[2.4.21~Folgerung]{Scharf_thesis} to conjugacy classes of prime order 
in wreath products. 
For general conjugacy classes in a wreath product, Scharf defined characters induced from linear characters on subgroups which are not necessarily centralizers. 

In Section~\ref{sec:gf_hLc_B} we construct higher Lie characters of type $B$. 
These characters (see Definition~\ref{def:HLC_Bn} below) are induced from linear characters of centralizers, and therefore satisfy Definition~\ref{def:hLc}(a). 
%
%
%
%
%
%
%
By the following theorem, they also satisfy Definition~\ref{def:hLc}(b).

\begin{theorem}\label{t:Bn-higher}
For any integer $k$ and positive integer $n$, the family $\{\psi_{B_n}^C \,:\, C\in \Conj(B_n)\}$ 
of $B_n$-characters from Definition~\ref{def:HLC_Bn} satisfies
\[
\roots_k^{B_n}
= \sum_{C \in \Conj(B_n) \,:\, C^k = \{1\}} \psi_{B_n}^C .
\]
\end{theorem}

See Theorem~\ref{t:roots_eq_psi} below. 
Theorem~\ref{t:Bn-higher} will be proved by comparing a generating function for root enumerators in $B_n$ (Theorem~\ref{t:gf_root_enumerator_Bn}) with a generating function for higher Lie characters in $B_n$ (Theorem~\ref{t:summation_4_Bn}). 
One deduces that, 
for all $n$ and $k$, the $k$-th root enumerator in $B_n$ is proper, providing a new proof of a result of Scharf~\cite[2.4.18 Satz]{Scharf_thesis}. 

Unfortunately, this approach is not directly applicable to root enumerators in $D_n$ since, for even $n>2$, $D_n$ has no family of higher Lie characters;  
see~\cite[Proposition 4.8.1]{Baddeley_thesis} and 
~\cite[Proposition 4 and p.\ 74, note added in proof]{Baddeley}.


To prove that root enumerators in $D_n$ are proper we therefore take a slightly different route.
We view $D_n$ as a subgroup of $B_n$ and prove the following.  

\begin{theorem}\label{t:Dn-higher}
Let $k$ be an integer and $n$ a positive integer. Then, for any $y \in B_n$:
\[
\sum_{C \in \Conj(B_n) \,:\, C^k = \{1\}} \psi_{B_n}^C(y) 
+ \sum_{C \in \Conj(B_n) \,:\, C^k = \{w_0\}} \psi_{B_n}^C(y) 
=
\begin{cases}
2 \roots_k^{D_n}(y), &\text{if } y \in D_n; \\
0, &\text{otherwise,}
\end{cases}
\]
where $w_0 := [-1,\ldots,-n]$ is the longest element in $B_n$ (a central involution).
\end{theorem}

See Theorem~\ref{t:sroots_eq_spsi} and Subsection~\ref{subsec:D_proper} below. 
Theorem~\ref{t:Dn-higher} implies Theorem~\ref{t:Dn-proper}. 

\begin{remark}\label{rem:intro_D}
By Theorem~\ref{t:Bn-higher}, the first sum on the LHS of the formula in Theorem~\ref{t:Dn-higher} is equal to $\roots_k^{B_n}(y)$. In other words, 
$\roots_k^{D_n}(y)$
is half the sum of 
$\roots_k^{B_n}(y)$ 
and a correcting term. 
When both $k$ and $n$ are even, there does not seem to be a direct connection between the two root enumerators. 
However, in the following two cases there is a simple connection: 
\begin{enumerate}
\item     
For $k$ even and $n$ odd, 
the second sum is zero. It follows that, in this case, $\roots_k^{D_n}(y) = \roots_k^{B_n}(y)/2$ for any $y \in D_n$. Indeed, in this case, for $x \in B_n$: $x^k = y$ if and only if $(x w_0)^k = y$, and $|\{x,xw_0\} \cap D_n| = 1$.

\item     
For odd $k$ and any $n$, a parity argument shows that 
$\roots_k^{D_n}$
is the restriction (to $D_n$) of
$\roots_k^{B_n}$.
Indeed, one can verify that, in this case, the two sums in Theorem~\ref{t:Dn-higher} are equal. 
\end{enumerate}


\end{remark}

Furthermore, if $n$ is odd, we can define higher Lie characters $\psi_{D_n}^C$ for all conjugacy classes $C$ in $D_n$; see Definition~\ref{def:psi_D}.
If $n$ is even then there is no family of induced characters that satisfies Definition~\ref{def:hLc}, but
Definition~\ref{def:psi_D} 
can still be used for odd $k$. 
We therefore have the following improvement of Theorem~\ref{t:Dn-higher}.

\begin{theorem}\label{t:Dn-higher_odd}
Let $k$ and $n \ge 1$ be integers which are not both even. 
Then
\[
\roots_k^{D_n} 
= \sum_{C \in \Conj(D_n) \,:\, C^k = \{1\}} \psi_{D_n}^C .
\]
\end{theorem}

See Subsection~\ref{subsec:HLC_Dn} below. 







\medskip

A {\em Gelfand model} of a finite group $G$ is a multiplicity-free sum of its irreducible characters. 
By the remarks above, if all the irreducible characters of $G$ come from real representations, then a Gelfand model for $G$ yields the character $\roots_2^G$. 
The problem of constructing Gelfand 
models for 
groups was introduced by Bernstein, Gelfand and Gelfand~\cite{BGG}. 
A construction of a Gelfand model for the symmetric group $S_n$, as a multiplicity-free sum of characters induced from centralizers of involutions, was presented by Klyachko~\cite{Klyachko} and by Inglis, Richardson and Saxl~\cite{Inglis}. Such a construction is called an {\em involution model}. 
Baddeley~\cite{Baddeley} generalized this construction to Weyl groups of types $B_n$ (for all $n$) and $D_n$ (for odd $n$), and noted~\cite[Proposition 4.8.1]{Baddeley_thesis} the non-existence of such a construction for Weyl groups of type $D_n$ when $n$ is even.  
Theorem~\ref{t:Dn-higher} can be used to construct a {\em double} Gelfand model for $D_n$, namely, a character of $D_n$ in which every irreducible character appears exactly twice, see Theorem~\ref{t:Dn_Gelfand} below. 


The rest of the paper is organized as follows. Section~\ref{sec:prel} contains necessary background and notation. 
In Section~\ref{sec:gf_roots_B} we compute 
a generating function (Theorem~\ref{t:gf_root_enumerator_Bn}) for root enumerators in $B_n$. 
Section~\ref{sec:gf_hLc_B} is devoted to higher Lie characters of type $B$, including an explicit definition (Definition~\ref{def:HLC_Bn}) and a proof of a generating function (Theorem~\ref{t:summation_4_Bn}) for their values. 
In Section~\ref{sec:proof_B} we compare the two generating functions and prove Theorem~\ref{t:Bn-higher}, restated as Theorem~\ref{t:roots_eq_psi}. 
In Section~\ref{sec:D_proper} we consider the groups $D_n$ and prove Theorems~\ref{t:Dn-higher}, 
\ref{t:Dn-proper}
and~\ref{t:Dn-higher_odd}. 
The paper concludes with applications and open problems regarding Gelfand models (Subsection~\ref{subsec:Gelfand}), 
index $2$ subgroups (Subsection~\ref{subsec:index2}),  
and related 
quasisymmetric functions (Subsection~\ref{subsec:QSF}).

\section*{Acknowledgments} 
The authors thank Axel Hultman, Eric Marberg  and Richard Stanley for useful discussions and references. They also thank the anonymous referee for his comments.

\section{Preliminaries}
\label{sec:prel}

\subsection{Basic notation}

Here and in the rest of the paper, $\ZZ_n$ denotes the cyclic group of order $n$ written multiplicatively, e.g., as the group of complex $n$-th roots of unity.

A character of a finite group is a class function which is the trace of a complex representation, namely: a nonnegative integer combination of the irreducible characters of the group. To emphasize that the multiplicities are nonnegative, we usually use the adjective {\em proper} for a character, contrasting it with a generalized (or virtual) character, for which the multiplicities of irreducible characters are arbitrary integers.

Induction and restriction of class functions are always denoted by arrows, and not merely by superscripts and subscripts. Other than that, all character theory notations are standard, using~\cite{Isaacs} as reference.


\subsection{The Weyl groups of types $A$, $B$ and $D$}
\label{subsec:prel_ABD}

The 
Weyl group of type $A$ and rank $n-1$ is isomorphic to the symmetric group $S_n$, i.e., the group of bijections from the set $[n]$ onto itself;
equivalently, the group of all $n \times n$ permutation matrices.
Similarly, the 
Weyl group of type $B$ and rank $n$ may be realized as the {\em group of signed permutations} $B_n$, consisting of all bijections $w$ of the set $[\pm n] := \{\pm1,\pm 2,\ldots,\pm n\}$ onto itself such that
\[
w(-i) = -w(i)
\qquad (1 \le i \le n),
\]
with composition as the group operation;
equivalently, the group of all $n \times n$ monomial matrices (namely, matrices having a unique nonzero entry in every row and column) such that the nonzero entries are all $\pm 1$, with matrix multiplication as the group operation.
This group is also known as the {\em hyperoctahedral group} of rank $n$. 
The 
Weyl group of type $D$ and rank $n$ may be realized as an index 2 subgroup of $B_n$, consisting all elements $w \in B_n$ for which the cardinality of the set $\{1 \le i \le n \,:\, w(i)<0\}$ (equivalently, the number of negative entries in the matrix) is even. 

For a partition $\lambda$ of $n$, let $\C_\lambda$ be the conjugacy class of $S_n$ consisting of all the permutations of cycle type $\lambda$.
Let $Z_\lambda$ be the centralizer of a permutation in $\C_\lambda$ (defined up to conjugacy). 
If $\lambda$ has $a_i$ parts equal to $i$ $(\forall\, i \ge 1)$, 
then $Z_\lambda$ is isomorphic to the direct product $\times_{i \ge 1} (\ZZ_i\wr S_{a_i})$.

In this paper we use an explicit description of the conjugacy classes 
in $B_n$ (but do not need the corresponding description in $D_n$). 
A {\em bipartition} $\blambda = (\lambda^+,\lambda^-)$ of a positive integer $n$, denoted $\blambda \vdash n$, is a pair of partitions of total size $n$. 
Conjugacy classes in $B_n$ 
are in one-to-one correspondence with bipartitions of $n$. 
Each element $w \in B_n$, viewed as a permutation of the set $[\pm n]$, can be written as a product of disjoint cycles of total length $2n$. Moreover, if $c = (a_1, a_2, \ldots, a_\ell)$ is such a cycle then so is $\bar{c} = (-{a}_1, -{a}_2, \ldots, -{a}_\ell)$, and either $c$ and $\bar{c}$ are disjoint or else $c = \bar{c}$. 
In the former case, the product $c \bar{c}$ is said to be a {\em positive cycle} of
$w$, of length $\ell$; 
otherwise $\ell$ is necessarily even and $c = \bar{c}$ is said to be a {\em negative cycle} of $w$, of length $\ell/2$. 
Then the pair of partitions $(\lambda^+, \lambda^-)$ for which the parts of $\lambda^+$ (respectively, $\lambda^-$) are the lengths of the positive (respectively, negative) cycles of $w$ is a bipartition of $n$, called the \emph{signed cycle type} of $w$. Two elements
of $B_n$ are conjugate if and only if they have the same signed cycle type. We denote by $C_\blambda$ the conjugacy class of elements of $B_n$ of signed cycle type $\blambda = (\lambda^+, \lambda^-)$. 

\subsection{Higher Lie characters of type $A$}
\label{subsec:hLC_type_A}

Recall the definition of {\em higher Lie characters} in $S_n$.

\begin{definition}\label{def:higher}
Let $\lambda$ be a partition of $n$, with $a_i$ parts equal to $i$ $(\forall\, i \ge 1)$.
For each $i$, let $\omega_i$ be the linear character on $\ZZ_i\wr S_{a_i}$ indexed by the $i$-tuple of partitions $(\emptyset, (a_i), \emptyset,\ldots,\emptyset)$. 
In other words, let $\zeta_i$ be a primitive irreducible character on the cyclic group $\ZZ_i$, and extend it to the wreath product $\ZZ_i\wr S_{a_i}$ so that it is homogeneous on the base subgroup $\ZZ_i^{a_i}$ and trivial on the wreathing subgroup $S_{a_i}$. Denote this extension by $\omega_i$. Note that this character does not depend on the choice of the complement $S_{a_i}$, because all the complements are conjugate, so $\ker(\omega_i)$ contains all of them.
Now let
\[
\omega^\lambda 
:= \bigotimes_{i=1}^d \omega_i,
\]
a linear character on the centralizer $Z_\lambda$. 
Define the corresponding {\em higher Lie character} to be the induced character
\[
\psi_{S_n}^\lambda 
:= \omega^\lambda\uparrow_{Z_\lambda}^{S_n}.
\]
The higher Lie character does not depend on the choice of $\zeta_i$ because the characters of $S_n$ are rational, thus invariant under any field automorphism.
\end{definition}

For an integer $k$, denote $\lambda \vdash_k n$ if all part sizes of $\lambda$ divide $k$. 
Scharf proved the following.

\begin{theorem}\label{t:roots_eq_psi_Sn}\cite{Scharf}
For any integer $k$ and nonnegative integer $n$,
\[
\roots^{S_n}_k
= \sum_{\lambda \vdash_k n} \psi_{S_n}^\lambda.
\]
\end{theorem}


\section{Generating function for root enumerators of type $B$}
\label{sec:gf_roots_B}

In this section we compute, for any integer $k$, an explicit generating function for the values of the $k$-th root enumerator $\roots_k^{B_n}$ on arbitrary elements of $B_n$. The result is expressed in terms of the $h$-th root enumerators $\roots_h^{\ZZ_2}$, corresponding to the cyclic group $\ZZ_2$ of order 2 (written multiplicatively as $\{+1,-1\}$) for various divisors $h$ of $k$.

We state the main result (Theorem~\ref{t:gf_root_enumerator_Bn}) in Subsection~\ref{subsec:roots_main}, and prove it in Subsection~\ref{subsec:roots_proof}.

\subsection{Main result}
\label{subsec:roots_main}

Let $\bmu = (\mu^+,\mu^-)$ be a bipartition of $n$. For each integer $j \ge 1$ and sign $\theta \in \ZZ_2 = \{+1,-1\}$, let $b_{j,\theta}$ be the number of parts of size $j$ in the partition $\mu^\theta$. Thus
\[
\sum_{j,\theta} j b_{j,\theta} = n.
\]
Let $\bt = (t_{j,\theta})_{j \ge 1, \theta \in \ZZ_2}$ be a countable set of indeterminates, and denote $\bt^{\bc(\bmu)} := \prod_{j,\theta} t_{j,\theta}^{j b_{j,\theta}}$.
Consider the ring $\CC[[\bt]]$ of formal power series in these indeterminates.
The main result of this section is the following.

\begin{theorem}\label{t:gf_root_enumerator_Bn}
For any integer $k$, the generating function for the $k$-th root enumerator in $B_n$ is
\[
\sum_{n \ge 0} 
\sum_{\bmu \vdash n} 
\roots_k^{B_n}(\bmu) \,
\frac{\bt^{\bc(\bmu)}}{|Z_\bmu|}
= \exp \left(
\sum_{j \ge 1} \sum_{\theta \in \ZZ_2}
\sum_{\substack{h|k \\ \gcd(h,j) = 1}} 
\roots_h^{\ZZ_2}(\theta) \,
\frac{t_{j,\theta}^{j k/h}}{2 j k/h}         \right) .
\]
\end{theorem}

\begin{remark}
The corresponding expression for the symmetric group $S_n$ was obtained by Lea\~nos, Moreno and Rivera-Mart\'inez~\cite{Leanos}.
In our notation, with the obvious modifications, their result is:
\[
\sum_{n \ge 0} 
\sum_{\mu \vdash n} 
\roots_k^{S_n}(\mu) \,
\frac{\bt^{\bc(\mu)}}{|Z_\mu|}
= \exp \left(
\sum_{j \ge 1}
\sum_{\substack{h|k \\ \gcd(h,j) = 1}} 
\frac{t_j^{j k/h}}{j k/h}
\right) .
\]
Formally, they use $g = k/h$ instead of our $h$, and sum over all $g \ge 1$ such that $\gcd(gj,k) = g$; see Remark 1 after the proof of~\cite[Proposition 2]{Leanos}. Their $t_j$ is our $t_j^j/j$.

\end{remark}

\subsection{Proof of Theorem~\ref{t:gf_root_enumerator_Bn}}
\label{subsec:roots_proof}

\begin{observation}\label{t:c_to_C}
Let $G$ be a finite group, and $k$ an arbitrary integer. Taking powers in $G$ induces a map between conjugacy classes, called the power map, as follows. Fix $g_0 \in G$, and let $h_0 := g_0^k$.
Assume that $g_0$ belongs to a conjugacy class $C_1 \in \Conj(G)$, and $h_0$ belongs to a conjugacy class $C_k \in \Conj(G)$.
Then:
\begin{enumerate}

\item 
The $k$-th power of every element of $C_1$ belongs to $C_k$.

\item
Every element of $C_k$ is the $k$-th power of some element of $C_1$.

\item
For every element $h \in C_k$,
\[
|\{g \in C_1 \,:\, g^k = h\}| 
= \frac{|C_1|}{|C_k|} 
= \frac{|Z_G(h_0)|}{|Z_G(g_0)|} .
\]

\end{enumerate} 
\end{observation}


\begin{lemma}\label{t:power_of_one_cycle_Bn} 
Assume that $x \in B_\ell$ consists of a single cycle, of length $\ell$ and $\ZZ_2$-class $\ve$.
Let $k$ be an integer, let $y := x^k \in B_\ell$, and denote
\[
d := \gcd(k,\ell).
\]
Then $y$ is a product of $d$ disjoint cycles, each of length $j := \ell/d$ (relatively prime to $k/d$) and $\ZZ_2$-class $\theta := \ve^{k/d}$.
\end{lemma}

\begin{proof}
Let $\ve_1, \ldots, \ve_\ell$ be a list of the nonzero matrix entries of $x$, in an order (of rows) corresponding to the order in the unique cycle of $x$ (with an arbitrary starting point).
Each nonzero entry of the matrix corresponding to $x^d$ has the form $\ve_{i} \ve_{i+1} \cdots \ve_{i+d-1}$ for some index $i$, with index summation computed modulo $\ell$. 
Since $d$ divides the cycle length $\ell$, the element $x^d \in B_\ell$ is a product of $d$ cycles, each of length $j := \ell/d$.
The $\ZZ_2$-class of each cycle is 
\[
\prod_{t=1}^{\ell/d} \prod_{s=1}^{d} \ve_{i + (t-1)d + (s-1)} = \prod_{u=1}^{\ell} \ve_u = \ve ,
\]
the $\ZZ_2$-class of $x$, where $i$ is a suitable index. 
Since $\gcd(j,k/d) = \gcd(\ell/d,k/d) = 1$, the element $x^k = (x^d)^{k/d}$ has the same cycle structure as $x^d$, and each of its cycles has $\ZZ_2$-class $\ve^{k/d}$.
\end{proof}

Consider now two countable sets of indeterminates, $\br = (r_{\ell,\ve})_{\ell \ge 1,\ve \in \ZZ_2}$ and $\bt = (t_{j,\theta})_{j \ge 1,\theta \in \ZZ_2}$.
We shall work in the ring $\CC[[\br,\bt]]$ of formal power series in these indeterminates, over the field $\CC$ of complex numbers.
Our first result gives a generating function for the number, refined by cycle type, of $k$-th roots of an element of $B_n$ which is a product of disjoint cycles of a fixed type, namely length $j \ge 1$ and $\ZZ_2$-class $\theta$.

\begin{lemma}\label{t:root_enumerator_similar_cycles_Bn} 
Fix $j \ge 1$ and $\theta \in \ZZ_2$.
For each $b \ge 0$, pick an element $y \in B_{j b}$ which is a product of $b$ disjoint cycles, each of length $j$ and $\ZZ_2$-class $\theta$. 
The generating function for the number of elements $x \in B_{j b}$ satisfying $x^k = y$ is
\begin{align*}
\sum_{b \ge 0} \frac{t_{j,\theta}^{j b}}{b! (2j)^b} 
\sum_{\blambda \vdash j b} \sum_{\substack{x \in C_\blambda \\ x^k = y}} 
\prod_{\ell,\ve} r_{\ell,\ve}^{\ell m_{\ell,\ve}} 
&= \exp \left( 
\sum_{\substack{h | k \\ \gcd(h,j) = 1}} 
\frac{t_{j,\theta}^{j k/h}}{2 j k/h} 
\sum_{\ve \in R_{h,\theta}} r_{j k/h,\ve}^{j k/h} 
\right) ,
\end{align*}
where $R_{h,\theta} := \{\ve \in \ZZ_2 \,:\, \ve^h = \theta\}$ and cycle-type $\blambda$ has $m_{\ell,\ve}$ cycles of length $\ell$ and $\ZZ_2$-class $\ve$.
\end{lemma}

\begin{proof}
By Observation~\ref{t:c_to_C}, applied to $B_n$, the number of possible elements $x \in B_n$ in a specific conjugacy class, satisfying $x^k = y$ for a specific $y \in B_n$, is a quotient of the sizes of two conjugacy classes, that of $x$ and that of $y$. Let us compute the sizes of these conjugacy classes.

In order to compute the size of the conjugacy class of $y$, first write an arbitrary permutation of the numbers $1, \ldots, j b$. Interpret the first $j$ numbers as forming a cycle, and assign to each of the numbers an element of $\ZZ_2$ such that the $\ZZ_2$-class of the whole cycle is $\theta$. This can be done in $|\ZZ_2|^{j-1}$ ways. Treat similarly the other cycles, each having length $j$ and $\ZZ_2$-class $\theta$. Noting that any permutation of the $b$ cycles, as well as an arbitrary cyclic shift of each cycle, yields the same element of $B_{j b}$, it follows that the size of the conjugacy class of $y$ in $B_{j b}$ is
\[
|C_y|
= \frac{(j b)! \cdot (2^{j-1})^b}{b! \cdot j^b} 
= \frac{(j b)! \cdot 2^{j b}}{b! \cdot (2j)^b} .
\]

According to Lemma~\ref{t:power_of_one_cycle_Bn}, for $y \in B_{j b}$ having $b$ cycles, each of length $j$ and $\ZZ_2$-class $\theta$,
each $x \in B_{j b}$ satisfying $x^k = y$ is a product of disjoint cycles of various lengths $\ell = d j$, for various divisors $d$ of $k$ satisfying $\gcd(k,\ell) = d$, or equivalently $\gcd(k/d,j) = 1$; and the $\ZZ_2$-class $\ve$ of each cycle satisfies $\ve^{k/d} = \theta$.
Denote
\[
D_k(j) := 
\{d \,:\, d \ge 1, \,\, d|k, \,\, \gcd(k/d,j) = 1\} ,
\]
and assume that $d_1, \ldots, d_q$ are the distinct elements of $D_k(j)$.
For each $1 \le i \le q$, let $m_i$ be the number of cycles of $x$ of length $\ell_i = d_i j$. 
The total number of cycles of $y$, each of length $j$, is then $m_1 d_1 + \ldots + m_q d_q = b$.
It follows that the cycle-length distributions of elements $x \in B_{j b}$ satisfying $x^k = y$ correspond to elements of the set
\[
M_k(j,b) := 
\{(m_1, \ldots, m_q) \in \ZZ_{\ge 0}^q \,:\, m_1 d_1 + \ldots + m_q d_q = b\} .
\]
The $\ZZ_2$-class $\ve$ of a cycle of length $d_i j$ must satisfy $\ve^{k/d_i} = \theta$.
For each integer $h$, let
\[
R_{h,\theta} 
:= \{\ve \in \ZZ_2 \,:\, \ve^h = \theta\},
\]
so that $|R_{h,\theta}| = \roots^{\ZZ_2}_h(\theta)$.

Assume now that $x$ has $m_{i,\ve}$ cycles of length $\ell_i = d_i j$ and $\ZZ_2$ class $\ve$ $(1 \le i \le q,\, \ve \in R_{k/d_i,\theta})$,
so that $\sum_\ve m_{i,\ve} = m_i$ $(1 \le i \le q)$.
A computation similar to the one for $y$ shows that the size of the conjugacy class of $x$ is
\[
|C_x|
= \frac{(j b)! \cdot 
\prod_{i=1}^{q} \prod_{\ve \in R_{k/d_i,\theta}} (2^{\ell_i - 1})^{m_{i,\ve}}} 
{\prod_{i=1}^{q} \prod_{\ve \in R_{k/d_i,\theta}} m_{i,\ve}! \cdot \ell_i^{m_{i,\ve}}} 
= (j b)! \cdot 
\prod_{i=1}^{q} \left( \frac{2^{\ell_i}}{2 \ell_i} \right)^{m_i}
\prod_{\ve \in R_{k/d_i,\theta}} \frac{1}{m_{i,\ve}!} .
\]
Dividing this by the size of the conjugacy class of $y$, and using $\sum_i m_i d_i = b$ (or equivalently $\sum_i m_i \ell_i = j $b) yields
\[
\frac{|C_x|}{|C_y|}
= b! (2j)^b \cdot 
\prod_{i=1}^{q} \frac{1}{(2 \ell_i)^{m_i}} 
\prod_{\ve \in R_{k/d_i,\theta}} \frac{1}{m_{i,\ve}!} .
\]
This is the number of elements $x' \in C_x$ satisfying $(x')^k = y$.
In order to record the cycle types of $x$ and $y$, let us introduce two countable sets of indeterminates, $\br = (r_{\ell,\ve})_{\ell \ge 1,\ve \in \ZZ_2}$ and $\bt = (t_{j,\theta})_{j \ge 1,\theta \in \ZZ_2}$. Multiply the above quantity by the monomial $\prod_{i,\ve} r_{\ell_i,\ve}^{\ell_i m_{i,\ve}}$
to get
\[
b! (2j)^b \cdot 
\prod_{i=1}^{q} \frac{1}{(2 \ell_i)^{m_i}} 
\prod_{\ve \in R_{k/d_i,\theta}} \frac{r_{\ell_i,\ve}^{\ell_i m_{i,\ve}}}{m_{i,\ve}!} .
\]
Sum this, for each $1 \le i \le q$, over all the decompositions of $m_i$ into a sum of nonnegative integers $m_{i,\ve}$ $(\ve \in R_{k/d_i,\theta})$, and use the multinomial identity
\[
\left( \sum_{\ve \in R} x_\ve \right)^m
= m! \cdot 
\sum_{\substack{m_\ve \ge 0 \, (\ve \in R) \\ \sum_{\ve} m_\ve = m}} \,\,
\prod_{\ve \in R} \frac{x_\ve^{m_\ve}}{m_\ve!} 
\]
for $x_\ve = r_{\ell_i,\ve}^{\ell_i}$, $m_\ve = m_{i,\ve}$, $m = m_i$ and $R = R_{k/d_i,\theta}$,
to get
\begin{align*}
&\ b! (2j)^b \cdot 
\prod_{i=1}^{q} \frac{1}{(2 \ell_i)^{m_i}} 
\sum_{\substack{m_{i,\ve} \ge 0 \, (\ve \in R_{k/d_i,\theta}) \\ \sum_{\ve} m_{i,\ve} = m_i}}
\prod_{\ve \in R_{k/d_i,\theta}} \frac{r_{\ell_i,\ve}^{\ell_i m_{i,\ve}}}{m_{i,\ve}!} \\
&= b! (2j)^b \cdot 
\prod_{i=1}^{q} \frac{1}{(2 \ell_i)^{m_i} m_i!} 
\left( 
\sum_{\ve \in R_{k/d_i,\theta}} r_{\ell_i,\ve}^{\ell_i} 
\right)^{m_i} .
\end{align*}
Now sum this over all $(m_1,\ldots,m_q) \in M_k(j,b)$, namely $m_i \ge 0$ such that $m_1 \ell_1 + \ldots + m_q \ell_q = j b$, to get
\[
b! (2j)^b \cdot 
\sum_{(m_1,\ldots,m_q) \in M_k(j,b)} 
\prod_{i=1}^{q} \frac{1}{m_i!} 
\left( 
\frac{1}{2 \ell_i} \sum_{\ve \in R^\theta_{k/d_i}} r_{\ell_i,\ve}^{\ell_i} 
\right)^{m_i} .
\]
Dividing by $b! (2j)^b$, the above sum is the coefficient of $t_{j,\theta}^{j b}$ in
\[
\prod_{i=1}^{q} \sum_{m_i \ge 0}
\frac{1}{m_i!} 
\left( 
\frac{t_{j,\theta}^{\ell_i}}{2 \ell_i} \sum_{\ve \in R_{k/d_i,\theta}} r_{\ell_i,\ve}^{\ell_i} 
\right)^{m_i} \\
= \prod_{i=1}^{q} 
\exp \left( 
\frac{t_{j,\theta}^{\ell_i}}{2 \ell_i} \sum_{\ve \in R_{k/d_i,\theta}} r_{\ell_i,\ve}^{\ell_i} 
\right) .
\]
By the above computations, this is (up to a factor $b! (2j)^b$) the generating function for all the elements $x \in B_{j b}$ satisfying $x^k = y$ for a specific element $y$, which has $b$ cycles of length $j$ and $\ZZ_2$-class $\theta$.
The solutions are counted by conjugacy class: If $x$ has $m_{\ell,\ve}$ cycles of length $\ell$ and $\ZZ_2$-class $\ve$ $(\ell \ge 1, \ve \in \ZZ_2)$, then it contributes $1$ to the coefficient of $t_{j,\theta}^{j b} \prod_{\ell,\ve} r_{\ell,\ve}^{\ell m_{\ell,\ve}}$.
The generating function includes summation over $b \ge 0$. Therefore, recovering the factor $b! (2j)^b$,
\[
\sum_{b \ge 0} \frac{1}{b! (2j)^b} 
\sum_{\blambda \vdash j b} \sum_{\substack{x \in C_\blambda \\ x^k = y}} t_{j,\theta}^{j b}
\prod_{\ell,\ve} r_{\ell,\ve}^{\ell m_{\ell,\ve}} 
= \prod_{i=1}^{q} 
\exp \left( 
\frac{t_{j,\theta}^{\ell_i}}{2 \ell_i} \sum_{\ve \in R_{k/d_i,\theta}} r_{\ell_i,\ve}^{\ell_i} 
\right) .
\]
Recall now that $d_i$ $(1 \le i \le q)$ are all the divisors of $k$ which satisfy $\gcd(k/d_i,j) = 1$, so that $h_i := k/d_i$ are the divisors of $k$ which satisfy $\gcd(h_i,j) = 1$. Also, $\ell_i = d_i j = j k/h_i$. Thus
\[
\sum_{b \ge 0} \frac{t_{j,\theta}^{j b}}{b! (2j)^b} 
\sum_{\blambda \vdash j b} \sum_{\substack{x \in C_\blambda \\ x^k = y}} 
\prod_{\ell,\ve} r_{\ell,\ve}^{\ell m_{\ell,\ve}} 
= \prod_{\substack{h | k \\ \gcd(h,j) = 1}} 
\exp \left( 
\frac{t_{j,\theta}^{j k/h}}{2 j k/h} \sum_{\ve \in R_{h,\theta}} r_{j k/h,\ve}^{j k/h} 
\right) ,
\]
as claimed.
\end{proof} 

The general case follows.

\begin{theorem}\label{t:gf_refined_root_enumerator_Bn}
For any integer $k$, the refined generating function for elements $x \in B_n$ satisfying $x^k = y$ is
\begin{align*}
&\ \sum_{n \ge 0} \frac{1}{|B_n|}
\sum_{\bmu \vdash n} \sum_{y \in C_\bmu} 
\sum_{\blambda \vdash n} 
\sum_{\substack{x \in C_\blambda \\ x^k = y}} 
\prod_{\ell,\ve} r_{\ell,\ve}^{\ell m_{\ell,\ve}} 
\prod_{j,\theta} t_{j,\theta}^{j b_{j,\theta}} \\
&= \exp \left( 
\sum_{j,\theta}
\sum_{\substack{h|k \\ \gcd(h,j) = 1}} 
\frac{t_{j,\theta}^{j k/h}}{2 j k/h} \sum_{\ve \in R_{h,\theta}} r_{j k/h,\ve}^{j k/h} 
\right) ,
\end{align*}
where $R_{h,\theta} := \{\ve \in \ZZ_2 \,:\, \ve^h = \theta\}$, cycle-type $\blambda$ has $m_{\ell,\ve}$ cycles of length $\ell$ and $\ZZ_2$-class $\ve$, and cycle-type $\bmu$ has $b_{j,\theta}$ cycles of length $j$ and $\ZZ_2$-class $\theta$. 
\end{theorem}

\begin{proof}
Let $y = \prod_{j,\theta} y_{j,\theta}$, where $y_{j,\theta}$ has $b_{j,\theta}$ cycles, each of length $j$ and $\ZZ_2$-class $\theta$. Use Lemma~\ref{t:root_enumerator_similar_cycles_Bn}, with $b = b_{j,\theta}$, and multiply over all $j$ and $\theta$ to get
\begin{align*}
&\ \prod_{j,\theta} 
\sum_{b_{j,\theta} \ge 0} \frac{t_{j,\theta}^{j b_{j,\theta}}}{b_{j,\theta}! (2j)^{b_{j,\theta}}} 
\sum_{\blambda \vdash j b_{j,\theta}} \sum_{\substack{x \in C_\blambda \\ x^k = y_{j,\theta}}} 
\prod_{\ell,\ve} r_{\ell,\ve}^{\ell m_{\ell,\ve}} \\
&= \prod_{j,\theta}
\prod_{\substack{h | k \\ \gcd(h,j) = 1}} 
\exp \left( 
\frac{t_{j,\theta}^{j k/h}}{2 j k/h} \sum_{\ve \in R_{h,\theta}} r_{j k/h,\ve}^{j k/h} 
\right) ,
\end{align*}
where $R_{h,\theta} := \{\ve \in \ZZ_2 \,:\, \ve^h = \theta\}$. 
Using
\[
|C_y| = \frac{|B_n|}{\prod_{j,\theta} b_{j,\theta}! (2j)^{b_{j,\theta}}}
\]
gives the equivalent form
\begin{align*}
&\ \sum_{n \ge 0} \frac{1}{|B_n|}
\sum_{\bmu \vdash n} \sum_{y \in C_\bmu} 
\sum_{\blambda \vdash n} 
\sum_{\substack{x \in C_\blambda \\ x^k = y}} 
\prod_{\ell,\ve} r_{\ell,\ve}^{\ell m_{\ell,\ve}} 
\prod_{j,\theta} t_{j,\theta}^{j b_{j,\theta}} \\
&= \prod_{j,\theta}
\prod_{\substack{h | k \\ \gcd(h,j) = 1}} 
\exp \left( 
\frac{t_{j,\theta}^{j k/h}}{2 j k/h} \sum_{\ve \in R_{h,\theta}} r_{j k/h,\ve}^{j k/h} 
\right) ,
\end{align*}
as claimed.
\end{proof}

\begin{proof}[Proof of Theorem~\ref{t:gf_root_enumerator_Bn}.] 
Set $r_{\ell,\ve} = 1$ in Theorem~\ref{t:gf_refined_root_enumerator_Bn}, for all $\ell$ and $\ve$, and note that 
\[
\sum_{\blambda \vdash n} 
\sum_{\substack{x \in C_\blambda \\ x^k = y}} 1
= \roots_k^{B_n}(y) .
\]
Recall that $|R_{h,\theta}| = \roots_h^{\ZZ_2}(\theta)$ and $|Z_\bmu| = |B_n|/|C_\bmu|$.
Denoting $\bt^{\bc(\bmu)} :=         \prod_{j,\theta} t_{j,\theta}^{j b_{j,\theta}}$, Theorem~\ref{t:gf_refined_root_enumerator_Bn} reduces to Theorem~\ref{t:gf_root_enumerator_Bn}, presenting the (unrefined) generating function for the $k$-th root enumerator in $B_n$.
\end{proof}

\section{Generating function for higher Lie characters of type $B$}
\label{sec:gf_hLc_B}

In this section we introduce higher Lie characters of type $B$. 
The main result here is an explicit generating function for the values of these characters (Theorem~\ref{t:summation_4_Bn}). 
We define higher Lie characters for $B_n$ in Subsection~\ref{subsec:def_hLc_B}, state the main theorem in Subsection~\ref{subsec:main_hLc_B}, and prove it in the following subsections.

A notational remark: In this section we do not consider $k$-th roots at all, and therefore the letter $k$ will have other uses.

\subsection{Definitions}
\label{subsec:def_hLc_B}

In this subsection we define a higher Lie character for any element (in fact, for any conjugacy class) in $B_n \cong \ZZ_2 \wr S_n$. 




\begin{lemma}\label{t:centralizer_in_Bn}
{\rm (Centralizers in $B_n$)}
\begin{enumerate}

\item[(a)]
Let $x \in B_n$ be an element of cycle type $\blambda = (\lambda^+,\lambda^-)$ where, for each $i \ge 1$ and $\ve \in \ZZ_2$, the partition $\lambda^\ve$ has $a_{i,\ve}$ parts of size $i$. 
Write 
\[
x = \prod_{i \ge 1} x_{i,+} x_{i,-}\,, 
\]
where $x_{i,\ve}$ has $a_{i,\ve}$ cycles of length $i$ and $\ZZ_2$-class $\ve$ $(i \ge 1,\, \ve \in \ZZ_2)$; of course, only finitely many factors here are nontrivial.
Then the centralizer 
$Z_x = Z_{B_n}(x)$ satisfies
\[
Z_{B_n}(x) = \bigtimes_{i \ge 1} \left( Z_{B_{ia_{i,+}}}(x_{i,+}) \times Z_{B_{ia_{i,-}}}(x_{i,-}) \right) ,
\]
and is therefore isomorphic to the direct product 
\[
\bigtimes_{i \ge 1} \left( G_{i,+} \wr S_{a_{i,+}} \times G_{i,-} \wr S_{a_{i,-}} \right) ,
\]
where $G_{i,\ve}$ is the centralizer in $B_i$ of a cycle of length $i$ and $\ZZ_2$-class $\ve$.
By convention, $G \wr S_0$ is the trivial group while $G \wr S_1 \cong G$. 

\item[(b)]
Let $x_{i,+} \in B_i$ consist of a single cycle, of length $i$ and class $+1 \in \ZZ_2$. Then the centralizer $ Z_{x_{i,+}} = Z_{B_i}(x_{i,+}) = G_{i,+}$ is isomorphic to the group $\ZZ_2 \times \ZZ_i$, where the generator of $\ZZ_i$ is $x_{i,+}$ and the generator of $\ZZ_2$ is the central {\rm (}longest{\rm )} element $w_0 = [-1,\ldots,-i] \in B_i$.
\item[(c)]
Let $x_{i,-} \in B_i$ consist of a single cycle, of length $i$ and class $-1 \in \ZZ_2$. Then the centralizer $Z_{x_{i,-}} = Z_{B_i}(x_{i,-}) = G_{i,-}$ is isomorphic to the cyclic group $\ZZ_{2i}$, with generator $x_{i,-}$; note that $x_{i,-}^i = w_0$.

\end{enumerate}
\end{lemma}

\begin{definition}\label{def:HLC_Bn}
{\rm (Higher Lie characters in $B_n$)}
Let $x$ be an element of cycle type $\blambda = (\lambda^+,\lambda^-)$ in $B_n$, as in Lemma~\ref{t:centralizer_in_Bn}(a).        
\begin{enumerate}

\item[(a)]
For each $i \ge 1$ and $\ve \in \ZZ_2$, let $\omega_{i,\ve}$ be the linear character on $G_{i,\ve} \wr S_{a_{i,\ve}}$ defined as follows: 
If $\ve = +1$ then, by Lemma~\ref{t:centralizer_in_Bn}(b), $G_{i,+} \cong \ZZ_2 \times \ZZ_i$.
Let $\omega_{i,+}$ be trivial on $\ZZ_2$, 
equal to a primitive irreducible character on the cyclic group $\ZZ_i$,
and trivial on the wreathing group $S_{a_{i,+}}$.
If $\ve = -1$ then, by Lemma~\ref{t:centralizer_in_Bn}(c), $G_{i.-} \cong \ZZ_{2i}$.
Let $\omega_{i,-}$ be equal to a primitive irreducible character on the cyclic group $\ZZ_{2i}$,
and trivial on the wreathing group $S_{a_{i,-}}$.
Let
\[
\omega^x 
:= 
\bigotimes_{i \ge 1} \left( \omega_{i,+} \otimes \omega_{i,-} \right),
\]
a linear character on $Z_x$. 

\item[(b)]
Define the corresponding {\em higher Lie character} to be the induced character
\[
\psi_{B_n}^x 
:= \omega^x 
\uparrow_{Z_x}^{B_n}. 
\]

\item[(c)]
It is easy to see that $\psi_{B_n}^x$ depends only on the conjugacy class $C$ (equivalently, the cycle type $\blambda$) of $x$, and can therefore be denoted $\psi_{B_n}^C$ or $\psi_{B_n}^\blambda$.

\end{enumerate}
\end{definition}

\subsection{Main result}
\label{subsec:main_hLc_B}

We use notation similar to that of Subsection~\ref{subsec:roots_main}. 
Specifically, let $\blambda = (\lambda^+,\lambda^-)$ and $\bmu = (\mu^+,\mu^-)$ be two bipartitions of $n$. 
For each integer $i \ge 1$ and sign $\ve \in \ZZ_2 = \{+1,-1\}$, let $a_{i,\ve}$ be the number of parts of size $i$ in the partition $\lambda^\ve$.
Similarly, for each integer $j \ge 1$ and sign $\theta \in \ZZ_2 = \{+1,-1\}$, let $b_{j,\theta}$ be the number of parts of size $j$ in the partition $\mu^\theta$. 
Thus
\[
\sum_{i,\ve} i a_{i,\ve} 
= \sum_{j,\theta} j b_{j,\theta} 
= n.
\]
Let $\bs = (s_{i,\ve})_{i \ge 1, \ve \in \ZZ_2}$ and $\bt = (t_{j,\theta})_{j \ge 1, \theta \in \ZZ_2}$ be two countable sets of indeterminates.
Denote $\bs^{\bc(\blambda)} := \prod_{i,\ve} s_{i,\ve}^{i a_{i,\ve}}$ and $\bt^{\bc(\bmu)} := \prod_{j,\theta} t_{j,\theta}^{j b_{j,\theta}}$.
Consider the ring $\CC[[\bs,\bt]]$ of formal power series in these indeterminates.
The main result of this section is the following.

\begin{theorem}\label{t:summation_4_Bn}
\begin{align*}
\sum_{n \ge 0}  
\sum_{\blambda \vdash n} \sum_{\bmu \vdash n} 
\psi_{B_n}^\blambda(\bmu) \,
\frac{\bs^{\bc(\blambda)} \bt^{\bc(\bmu)}}{|Z_\bmu|} 
&= \exp \left( \sum_{i,\ve} \sum_{j,\theta} \sum_{e | \gcd(i,j)}
K_{\ve,\theta}(e) \,
\frac{(s_{i,\ve} t_{j,\theta})^{ij/e}}{2ij/e} \right) .
\end{align*}
Here, for $\ve, \theta \in \ZZ_2 = \{+1,-1\}$ and $e \ge 1$,
\[
K_{\ve,\theta}(e)
:= \ve \theta \cdot \mu(2e)
+ \frac{(1+\ve)(1+\theta)}{2} \cdot \mu(e),
\] 
where $\mu(\cdot)$ is the usual (arithmetic) M\"obius function.
\end{theorem}

We prove this result in the following subsections. 

\subsection{Values of induced characters}

Let us start our computations 
by writing a general formula (Lemma~\ref{t:HLC_values}) for the values of each higher Lie character, as an induced character.

For an element $x\in B_n$ denote the conjugacy class and the centralizer of $x$ by $C_x$ and $Z_x$, respectively. 
Recall the definitions of $\omega^x$ and $\psi_{B_n}^x$ from Subsection~\ref{subsec:def_hLc_B}. 

\begin{lemma}\label{t:HLC_values}
If $x \in B_n$ 
then
\[
\psi_{B_n}^x(g) 
= \frac{|\C_x|}{|\C_g|} 
\sum_{z \in \C_g \cap Z_x} 
\omega^x(z) 
\qquad (\forall g \in B_n).
\]
\end{lemma}

\begin{proof}
Let $G$ be a group, and $\chi$ a character of a subgroup $H$ of $G$. 
Define a function $\chi^0: G \to \CC$ by
\[
\chi^0(g) :=
\begin{cases}
\chi(g), &\text{if } g \in H; \\
0, &\text{if } g \in G \setminus H.
\end{cases}
\]
By~\cite[(5.1)]{Isaacs}, an explicit formula for the induced character $\chi\uparrow_H^G$ is
\[
\chi\uparrow_H^G(g)
= \sum_{G = \cup_a aH} \chi^0(a^{-1} g a)
= \frac{1}{|H|} \sum_{a\in G} \chi^0(a^{-1} g a)
\qquad (\forall g \in G) .
\]
The mapping $f : G \to \C_g$ defined by
$f(a) := a^{-1} g a$ $(\forall a \in G)$ is surjective, and satisfies:
$f(a_1) = f(a_2)$ if and only if $a_1 a_2^{-1} \in Z_g$.
Hence
\[
\chi\uparrow_H^G(g)
= \frac{|Z_g|}{|H|} \sum_{z \in \C_g}\chi^0(z)
= \frac{|Z_g|}{|H|} \sum_{z \in \C_g \cap H} \chi(z).
\]
Consider now an element $x \in B_n$, 
and apply the above formula
with $G = B_n$, $H = Z_x = Z_{B_n}(x)$, and 
$\chi = \omega^x$, the linear character on the centralizer $Z_x$ described in Definition~\ref{def:HLC_Bn}(a). Then
\[
\psi_{B_n}^x(g) 
= \omega^x \uparrow_{Z_x}^{B_n} (g)
= \frac{|Z_g|}{|Z_x|} \sum_{z \in \C_g \cap Z_x} \omega^x(z) 
\qquad (\forall g \in B_n).
\]
Recalling that $|Z_x| = |B_n|/|\C_x|$ and  $|Z_g| = |B_n|/|\C_g|$ completes the proof.  
\end{proof}

\subsection{The structure of a single cycle}

We want to study the $B_n$-structure of elements $z \in \C_g \cap Z_x$.
Our main initial result is Corollary~\ref{t:cycles2_Bn}, describing the $B_n$-structure of a single cycle of $z$.

Assume that $x \in B_n$ has cycle type $\blambda = (\lambda^+,\lambda^-)$.
Following Lemma~\ref{t:centralizer_in_Bn}(a),
decompose
\[
x = \prod_{i \ge 1} x_{i,+} x_{i,-} \,,    
\]
where each $x_{i.\ve} \in B_n$ is a product of $a_{i,\ve}$ cycles of length $i$ and $\ZZ_2$-class $\ve$ $(i \ge 1, \ve \in \ZZ_2)$, and only finitely many of the factors are nontrivial. 
Then, by Lemma~\ref{t:centralizer_in_Bn},
\[
Z_{B_n}(x)  
= \bigtimes_{i \ge 1} \left( Z_{B_{i a_{i,+}}}(x_{i,+}) \times Z_{B_{i a_{i,-}}}(x_{i,-}) \right)
\cong \bigtimes_{i \ge 1} \left( G_{i,+} \wr S_{a_{i,+}} \times G_{i,+} \wr S_{a_{i,-}} \right),
\]
where $G_{i,+} \cong \ZZ_2 \times \ZZ_i$
and $G_{i,-} \cong \ZZ_{2i}$.

Assume that $g \in B_n$ has cycle type $\bmu = (\mu^+,\mu^-)$, with $b_{j,\theta}$ cycles of length $j$ and $\ZZ_2$-class $\theta$ $(j \ge 1, \theta \in \ZZ_2)$.
Let $z \in \C_g \cap Z_x$, and decompose it as
\[
z = \prod_{i \ge 1} z_{i,+} z_{i,-} \,,
\]
where $z_{i,\ve} \in Z_{B_{i a_{i,\ve}}}(x_{i,\ve}) \cong G_{i,\ve} \wr S_{a_{i,\ve}}$ $(i \ge 1, \ve \in \ZZ_2)$.
Using a finer decomposition, assume that $z_{i,\ve}$, as an element of $B_{i a_{i,\ve}}$, has $m_{i,\ve,j,\theta}$ cycles of length $j$ and $\ZZ_2$-class $\theta$ $(j \ge 1, \theta \in \ZZ_2)$.
Of course, 
\[
\sum_{j,\theta} j m_{i,\ve,j,\theta} = i a_{i,\ve}
\quad \text{ and } \quad 
\sum_{i,\ve} m_{i,\ve,j,\theta} = b_{j,\theta}.
\]
Assume now that, as an element of $G_{i,\ve} \wr S_{a_{i,\ve}}$, $z_{i,\ve}$ has a cycle $c$ of length $\ell$ and $G_{i,\ve}$-class $\gamma$.
What is the structure of $c$ as an element of $B_{i a_{i,\ve}}$?

Let $\zeta_i \in \CC$ be a primitive complex $i$-th root of $1$, generating the cyclic group $\ZZ_i$.

\begin{lemma}\label{t:cycles_from_GwrS_to_B}
Fix $i \ge 1$ and $\ve \in \ZZ_2$, and let $x_{i,\ve} \in B_{i a_{i,\ve}}$ be a product of $a_{i,\ve}$ cycles, each of length $i$ and $\ZZ_2$-class $\ve$.
Let $c \in 
Z_{B_{i a_{i,\ve}}}(x_{i,\ve}) 
\cong G_{i,\ve} \wr S_{a_{i,\ve}}$ be a cycle of length $\ell$ {\rm (}in $S_{a_{i,\ve}}${\rm )} and class $\gamma \in G_{i,\ve}$.
Then there are two possible cases:
\begin{enumerate}

\item[(a)]
If $\ve = +1$ and $\gamma = (\delta, \zeta_i^k) \in \ZZ_2 \times \ZZ_i \cong G_{i,+}$, let $d := \gcd(k,i)$.
Then, as an element of $B_{ia_{i,+}}$, $c$ is a product of $d$ disjoint cycles, each of length $j = \ell i/d$ and $\ZZ_2$-class $\theta = \delta^{i/d}$.

\item[(b)]
If $\ve = -1$ and $\gamma = \zeta_{2i}^k \in \ZZ_{2i} \cong G_{i,-}$, let $d := \gcd(k,i)$.
Then, as an element of $B_{ia_{i,-}}$, $c$ is a product of $d$ disjoint cycles, each of length $j = \ell i/d$ and $\ZZ_2$-class $\theta = (-1)^{k/d}$.

\end{enumerate}
\end{lemma}

\begin{proof}
Following Lemma~\ref{t:centralizer_in_Bn}, 
and after an appropriate conjugation in $B_{i a_{i,\ve}}$,
we can assume that we have the following scenario:
\begin{itemize}

\item 
There is an $\ell \times i$ rectangular array $(p_{s,t})$ of distinct integers in $[i a_{i,\ve}]$ (say, $p_{s,t} = (s-1)i + t$ for all $1 \le s \le \ell$ and $1 \le t \le i$).
Each of the $\ell$ rows corresponds to a cycle (of length $i$ and $\ZZ_2$-class $\ve$) of $x_{i,\ve}$. As an element of $B_{i a_{i,\ve}}$,
\begin{align*}
	x_{i,\ve} : \quad
	p_{s,t} &\mapsto p_{s,t+1}
	\qquad (1 \le s \le \ell,\, 1 \le t \le i),
\end{align*}
where $p_{s,i+1}$ is interpreted as $\ve \cdot p_{s,1}$. On the other elements of $[i a_{i,\ve}]$, $x_{i,\ve}$ acts as the identity.

\item 
As an element of $B_{i a_{i,\ve}}$, $c \in Z_{x_{i,\ve}}$ permutes the $\ell$ rows cyclically, and specifically acts as follows:
\begin{enumerate}
	
	\item[(a)]
	If $\ve = +1$ and 
	$\gamma = (\delta, \zeta_i^k) \in \ZZ_2 \times \ZZ_i$ $(0 \le k < i)$, then 
	\begin{align*}
		c : \quad
		p_{s,t} &\mapsto p_{s+1,t} 
		\qquad (1 \le s \le \ell-1,\, 1 \le t \le i), \\
		p_{\ell,t} &\mapsto \delta \cdot p_{1,t+k}
		\qquad (1 \le t \le i),
	\end{align*}
	with $p_{1,t+k}$ interpreted as $p_{1,t+k-i}$ if $t+k > i$.
	
	\item[(b)]
	If $\ve = -1$ and
	$\gamma = \zeta_{2i}^k \in \ZZ_{2i}$ $(0 \le k < 2i)$, then
	\begin{align*}
		c : \quad
		p_{s,t} &\mapsto p_{s+1,t} 
		\qquad (1 \le s \le \ell-1,\, 1 \le t \le i), \\
		p_{\ell,t} &\mapsto p_{1,t+k} 
		\qquad (1 \le t \le i),
	\end{align*}
	with $p_{1,t+k}$ interpreted as $-p_{1,t+k-i}$ if $t+k > i$ (thus as $p_{1,t+k-2i}$ if $t+k > 2i$).
	
\end{enumerate}

\end{itemize}
Since $c$ permutes the $\ell$ rows cyclically, each cycle of $c$ (as an element of $B_{i a_{i,\ve}}$) has length divisible by $\ell$. Concerning $c^\ell$:
\begin{enumerate}

\item[(a)]
If $\ve = +1$ and 
$\gamma = (\delta, \zeta_i^k) \in \ZZ_2 \times \ZZ_i$ $(0 \le k < i)$, then 
\begin{align*}
	c^\ell : \quad
	p_{s,t} &\mapsto \delta \cdot p_{s,t+k} 
	\qquad (\forall s,t), 
\end{align*}
with $p_{s,t+k}$ interpreted as $p_{s,t+k-i}$ if $t+k > i$.

\item[(b)]
If $\ve = -1$ and
$\gamma = \zeta_{2i}^k \in \ZZ_{2i}$ $(0 \le k < 2i)$, then
\[
c^\ell : \quad
p_{s,t} \mapsto p_{s,t+k} 
\qquad (\forall s,t),
\]
with $p_{s,t+k}$ interpreted as $-p_{s,t+k-i}$ if $t+k > i$, etc.

\end{enumerate}
Defining $d := \gcd(k,i)$, we have $\gcd(k/d,i/d) = 1$. The smallest multiple of $k$ which is divisible by $i$ is thus $(i/d) \cdot k = i \cdot (k/d)$. It follows that all the cycles of $c$ have length $\ell i/d$, and we have the following:
\begin{enumerate}

\item[(a)]
If $\ve = +1$ and 
$\gamma = (\delta, \zeta_i^k) \in \ZZ_2 \times \ZZ_i$ $(0 \le k < i)$, then 
\begin{align*}
	c^{\ell i/d} : \quad
	p_{s,t} &\mapsto \delta^{i/d} \cdot p_{s,t + (i/d) \cdot k} 
	= \delta^{i/d} \cdot p_{s,t} 
	\qquad (\forall s,t).
\end{align*}

\item[(b)]
If $\ve = -1$ and
$\gamma = \zeta_{2i}^k \in \ZZ_{2i}$ $(0 \le k < 2i)$, then
\begin{align*}
	c^{\ell i/d} : \quad
	p_{s,t} &\mapsto p_{s,t + (i/d) \cdot k}
	= p_{s,t + (k/d) \cdot i}
	= (-1)^{k/d} \cdot p_{s,t} 
	\qquad (\forall s,t).
\end{align*}

\end{enumerate}
This determines the $\ZZ_2$-class of each cycle of $c$. 
\end{proof}

Denoting $e := i/d$, so that $j = \ell e$, we can restate Lemma~\ref{t:cycles_from_GwrS_to_B} as follows.

\begin{corollary}\label{t:cycles2_Bn}
Fix $i,j \ge 1$ and $\ve,\theta \in \ZZ_2$, and let $x_{i,\ve} \in B_{i a_{i,\ve}}$ be a product of $a_{i,\ve}$ cycles, each of length $i$ and $\ZZ_2$-class $\ve$.
To each cycle $c$ of $z_{i,\ve} \in Z_{x_{i,\ve}} \cong G_{i,\ve} \wr S_{a_{i,\ve}}$ which contributes to $m_{i,\ve,j,\theta}$ there corresponds a common divisor $e$ of $i$ and $j$ such that, denoting $d := i/e$ and $\ell := j/e$, the following holds:
The cycle $c$ has length $\ell$ {\rm (}in $S_{a_{i,\ve}}${\rm )} and class $\gamma \in G_{i,\ve}$, and it corresponds {\rm (}as an element of $B_{i \ell} = B_{j d}${\rm )} to a product of $d$ disjoint cycles, each of length $j$ and class $\theta \in \ZZ_2$, where either
\begin{enumerate}

\item[(a)]
$\ve = +1$,
$\gamma = (\delta, \zeta_i^k) \in \ZZ_2 \times \ZZ_i \cong G_{i,+}$, 
$\gcd(k,i) = d$ and $\theta = \delta^{i/d}$, 
or

\item[(b)]
$\ve = -1$,
$\gamma = \zeta_{2i}^k \in \ZZ_{2i} \cong G_{i,-}$, 
$\gcd(k,i) = d$ and $\theta = (-1)^{k/d}$.

\end{enumerate}
\end{corollary}

\subsection{Summation on a single cycle}

We now want to 
compute the sum in Lemma~\ref{t:HLC_values} on a certain small subset of $\C_g \cap Z_x$, when $g$ and $x$ have special $B_n$ cycle types, with all cycles of the same length and $\ZZ_2$-class.
The main result here is Lemma~\ref{t:summation_1_Bn}, addressing summation over elements with a single underlying cycle.

As a computational tool, recall the following well-known fact regarding the classical M\"obius function $\mu : \NN \to \{-1,0,1\}$, defined by
$\mu(1) := 1$, $\mu(n) := (-1)^k$ if $n$ is a product of $k \ge 1$ distinct primes, and $\mu(n) := 0$ otherwise (namely, if $n$ is not square-free).

\begin{lemma}\label{t:Mobius}
For any positive integer $n$,
\[
\sum_{\substack{0 \le k < n \\ \gcd(k,n) = 1}} \zeta_n^k = \mu(n) \, .
\]    
\end{lemma}

Here are a few similar formulas that will also prove useful.

\begin{lemma}\label{t:Mobius2}
Fix a positive integer $n$. Then:
\begin{enumerate}
\item[(a)]
\[
\sum_{\substack{0 \le k < 2n \\ \gcd(k,n) = 1 \\ k \text{ odd}}} \zeta_{2n}^k 
= \mu(2n)
= \begin{cases}
	-\mu(n), &\text{if } n \text{ is odd;} \\
	0, &\text{if } n \text{ is even.}  
\end{cases}
\]
\item[(b)]
\[
\sum_{\substack{0 \le k < 2n \\ \gcd(k,n) = 1 \\ k \text{ even}}} \zeta_{2n}^k 
= -\mu(2n)
= \begin{cases}
	\mu(n), &\text{if } n \text{ is odd;} \\
	0, &\text{if } n \text{ is even.}  
\end{cases}
\]
\item[(c)]
\[
\sum_{\substack{0 \le k < 2n \\ \gcd(k,n) = 1}} \zeta_{2n}^k 
= 0 \, .
\]
\end{enumerate}
\end{lemma}

\begin{proof}
Let $n$ be a positive integer.
\begin{enumerate}
\item[(a)]
``$\gcd(k,n) = 1$ and $k$ odd'' is equivalent to $\gcd(k,2n) = 1$. Therefore, by Lemma~\ref{t:Mobius}, 
\[
\sum_{\substack{0 \le k < 2n \\ \gcd(k,n) = 1 \\ k \text{ odd}}} \zeta_{2n}^k
= \sum_{\substack{0 \le k < 2n \\ \gcd(k,2n) = 1}} \zeta_{2n}^k
= \mu(2n) \, .
\]
\item[(b)]
If $n$ is even then ``$k$ even'' contradicts ``$\gcd(k,n) = 1$'', and therefore the sum is zero. In that case ($n$ even) we also have $\mu(2n) = 0$, since $2n$ is not square-free.
Assume now that $n$ is odd.
If $k$ is even then $\gcd(k,n) = 1$ is equivalent to $\gcd(k/2,n) = 1$, and letting $k' := k/2$ yields
\[
\sum_{\substack{0 \le k < 2n \\ \gcd(k,n) = 1 \\ k \text{ even}}} \zeta_{2n}^k 
= \sum_{\substack{0 \le k' < n \\ \gcd(k',n) = 1}} \zeta_{2n}^{2k'}
= \mu(n) \, .
\]
In that case ($n$ odd) we also have $\mu(2n) = -\mu(n)$.
\item[(c)]
Add the formulas in (a) and (b).
\qedhere
\end{enumerate}
\end{proof}

The following result computes the sum in Lemma~\ref{t:HLC_values} only on the elements of $C_g \cap Z_x$ corresponding to one specific cycle in $S_{a_{i,\ve}}$; $g$ and $x$ are assumed to have $B_n$ cycle types with all cycles of the same length and $\ZZ_2$-class.

\begin{lemma}\label{t:summation_1_Bn}
Fix $i,j \ge 1$ and $\ve,\theta \in \ZZ_2$, and let $x = x_{i,\ve} \in B_{i a_{i,\ve}}$ be a product of $a_{i,\ve}$ cycles, each of length $i$ and $\ZZ_2$-class $\ve$.
Let $e$ be a common divisor of $i$ and $j$, and denote $d := i/e$ and $\ell := j/e$.
Let $\sigma \in S_{a_{i,\ve}}$ be a permutation which has a single cycle of length $\ell$, and is the identity outside the support of this cycle.
Let $R_{i,\ve,j,\theta}(\sigma,e)$ be the set of all the elements $z \in Z_x \cong G_{i,\ve} \wr S_{a_{i,\ve}}$ corresponding to the underlying permutation (cycle) $\sigma$, with suitable $G_{i,\ve}$-classes, such that, as elements of $B_{i \ell} = B_{j d}$, they are products of $d$ disjoint cycles, each of length $j$ and $\ZZ_2$-class $\theta$. 
Denote
\[
K_{i,\ve,j,\theta}(\sigma,e) 
:= \frac{1}{(2i)^{\ell - 1}} \sum_{z \in R_{i,\ve,j,\theta}(\sigma,e)} \omega^x(z) \, .
\]    
Then $K_{i,\ve,j,\theta}(\sigma,e) = K_{\ve,\theta}(e)$ actually depends only on $\ve$, $\theta$ and $e$. Specifically:
\begin{align*}
K_{+,+}(e) 
&= \mu(2e) + 2 \cdot \mu(e)
= \begin{cases}
	\mu(e), &\text{if } e \text{ is odd;} \\
	2 \cdot \mu(e), &\text{if } e \text{ is even;} 
\end{cases} \\
K_{+,-}(e) 
&= -\mu(2e)
= \begin{cases}
	\mu(e), &\text{if } e \text{ is odd;} \\
	0, &\text{if } e \text{ is even.}
\end{cases} \\
K_{-,+}(e) 
&= -\mu(2e)
= \begin{cases}
	\mu(e), &\text{if } e \text{ is odd;} \\
	0, &\text{if } e \text{ is even.}  
\end{cases} \\
K_{-,-}(e) 
&= \mu(2e)
= \begin{cases}
	-\mu(e), &\text{if } e \text{ is odd;} \\
	0, &\text{if } e \text{ is even.}  
\end{cases}
\end{align*}
Written more compactly, for any $\ve, \theta \in \ZZ_2 = \{+1,-1\}$ and $e \ge 1$:
\[
K_{\ve,\theta}(e)
:= \ve \theta \cdot \mu(2e)
+ \frac{(1+\ve)(1+\theta)}{2} \cdot \mu(e).
\] 
\end{lemma}

\begin{proof}
By 
Corollary~\ref{t:cycles2_Bn},
the set of possible $G_{i,\ve}$-classes $\gamma$ of elements $z \in R_{i,\ve,j,\theta}(\sigma,e)$ depends on $i$, $\ve$, $d = i/e$ and $\theta$, but not on $j$ or $\sigma$ (as long as $j$ is a multiple of $e$).
Denote this set by $C_{\ve,\theta}(i,d)$.
Then:
\begin{align*}
C_{+,+}(i,d) &=
\begin{cases}
	\{(+1,\zeta_i^k) \,:\, \gcd(k,i) = d\}, &\text{if } i/d \text{ is odd;} \\
	\{(\pm 1,\zeta_i^k) \,:\, \gcd(k,i) = d\}, &\text{if } i/d \text{ is even;}         \end{cases} \\
C_{+,-}(i,d) &=
\begin{cases}
	\{(-1,\zeta_i^k) \,:\, \gcd(k,i) = d\}, &\text{if } i/d \text{ is odd;} \\
	\varnothing, &\text{if } i/d \text{ is even;} 
\end{cases} \\ 
C_{-,+}(i,d) &=
\{\zeta_{2i}^k \,:\, \gcd(k,i) = d,\, k/d \text{ even}\}; \\
C_{-,-}(i,d) &=
\{\zeta_{2i}^k \,:\, \gcd(k,i) = d,\, k/d \text{ odd}\}. 
\end{align*}
The number of elements $z \in R_{i,\ve,j,\theta}(\sigma,e)$ with any specific $G_{i,\ve}$-class is $|G_{i,\ve}|^{\ell - 1} = (2i)^{\ell - 1}$.
Denote 
\[
K_{i,\ve,j,\theta}(\sigma,e) 
:= \frac{1}{(2i)^{\ell - 1}} \sum_{z \in R_{i,\ve,j,\theta}(\sigma,e)} \omega^x(z) \, .
\]
It follows, by Definition~\ref{def:HLC_Bn}(a) of $\omega^x$, that
\[
K_{i,\ve,j,\theta}(\sigma,e)
= \begin{cases}
\sum\limits_{\gamma = (\delta,\zeta_i^k) \in C_{+,\theta}(i,d)} \zeta_i^k,
&\text{if } \ve = +1; \\
& \\
\sum\limits_{\gamma = \zeta_{2i}^k \in C_{-,\theta}(i,d)} \zeta_{2i}^k,
&\text{if } \ve = -1.
\end{cases}
\]
Consider, for example, the case of $\ve = \theta = +1$ and $i/d$ odd:
\[
K_{i,+,j,+}(\sigma,e)
= \sum_{\gamma = (\delta,\zeta_i^k) \in C_{+,+}(i,d)} \zeta_i^k 
= \sum_{\substack{\gamma = (+1,\zeta_i^k) \\ 
	0 \le k < i \\ \gcd(k,i) = d}} \zeta_i^k
	= \sum_{\substack{0 \le k < i \\ \gcd(k,i) = d}} \zeta_i^k.
	\]
	Denoting $k' := k/d$ and using Lemma~\ref{t:Mobius}, we conclude that
	\[
	\sum_{\substack{0 \le k < i \\ \gcd(k,i) = d}} \zeta_i^k
	= \sum_{\substack{0 \le k' < i/d \\ \gcd(k',i/d) = 1}} \zeta_{i/d}^{k'}
	= \mu(i/d) \, .
	\]
	A similar computation applies to all the cases with $\ve = +1$, yielding
	\[
	K_{i,+,j,+}(\sigma,e)
	= \begin{cases}
\mu(i/d), &\text{if } i/d \text{ is odd;} \\
2 \cdot \mu(i/d), &\text{if } i/d \text{ is even} 
\end{cases}
\]
and
\[
K_{i,+,j,-}(\sigma,e) 
= \begin{cases}
\mu(i/d), &\text{if } i/d \text{ is odd;} \\
0, &\text{if } i/d \text{ is even.}
\end{cases}
\]
The cases with $\ve = -1$ are slightly more involved.
Consider, for example, the case of $\ve = -1$ and $\theta = +1$:
\[
K_{i,-,j,+}(\sigma,e)
= \sum_{\gamma = \zeta_{2i}^k \in C_{-,+}(i,d)} \zeta_{2i}^k 
= \sum_{\substack{0 \le k < 2i \\ \gcd(k,i) = d \\ k/d \text{ even}}} \zeta_{2i}^k \, .
\]
Denoting $k' := k/d$ 
and using Lemma~\ref{t:Mobius2}(b), 
we obtain
\[
\sum_{\substack{0 \le k < 2i \\ \gcd(k,i) = d \\ k/d \text{ even}}} \zeta_{2i}^k
= \sum_{\substack{0 \le k' < 2i/d \\ \gcd(k',i/d) = 1 \\ k' \text{ even}}} \zeta_{2i/d}^{k'} 
= -\mu(2i/d)
= \begin{cases}
\mu(i/d), &\text{if } i/d \text{ is odd;} \\
0, &\text{if } i/d \text{ is even.}  
\end{cases}
\]
Finally, for the case of $\ve = \theta = -1$:
\[
K_{i,-,j,-}(\sigma,e)
= \sum_{\gamma = \zeta_{2i}^k \in C_{-,-}(i,d)} \zeta_{2i}^k 
= \sum_{\substack{0 \le k < 2i \\ \gcd(k,i) = d \\ k/d \text{ odd}}} \zeta_{2i}^k \, .
\]
Denoting $k' := k/d$ 
and using Lemma~\ref{t:Mobius2}(a), 
we obtain
\[
\sum_{\substack{0 \le k < 2i \\ \gcd(k,i) = d \\ k/d \text{ odd}}} \zeta_{2i}^k
= \sum_{\substack{0 \le k' < 2i/d \\ \gcd(k',i/d) = 1 \\ k' \text{ odd}}} \zeta_{2i/d}^{k'} 
= \mu(2i/d)
= \begin{cases}
-\mu(i/d), &\text{if } i/d \text{ is odd;} \\
0, &\text{if } i/d \text{ is even.}  
\end{cases}
\]
Replacing $i/d$ by $e$, everywhere, gives formulas involving $\mu(e)$ which depend on the parity of $e$. Using, in addition, the fact that
\[
\mu(2e)
= \begin{cases}
-\mu(e), &\text{if } e \text{ is odd;} \\
0, &\text{if } e \text{ is even}  
\end{cases}
\]
gives alternative formulas, independent of the parity of $e$. Using the signs $\ve$ and $\theta$ explicitly finally leads to the compact formula
\[
K_{\ve,\theta}(e)
:= \ve \theta \cdot \mu(2e)
+ \frac{(1+\ve)(1+\theta)}{2} \cdot \mu(e),
\] 
as claimed.
\end{proof}

\subsection{Proof of Theorem~\ref{t:summation_4_Bn}}

Extending the previous result, we shall now sum $\omega^x$ on the whole set $C_g \cap Z_x$, with increasing generality of the cycle types of $g$ and $x$. 
This will lead to a proof of Theorem~\ref{t:summation_4_Bn}, providing a generating function for the values of higher Lie characters $\psi_{B_n}^\blambda = \psi_{B_n}^x$.

\begin{definition}
For positive integers $i$ and $j$, let
\[
E(i,j) := \{e \ge 1 \,:\, e \text{ divides both } i \text{ and } j\}.
\]
\end{definition}

Note that $E(i,j)$ is never empty, since it always contains $e = 1$.

For an indeterminate $s$, let $\CC[[s]]$ be the ring of formal power series in $s$ over the field $\CC$.
We now extend Lemma~\ref{t:summation_1_Bn}, and compute the sum in Lemma~\ref{t:HLC_values} on the whole set $C_g \cap Z_x$, still restricting $g$ and $x$ to have cycle types with all cycles of the same length and $\ZZ_2$-class.

\begin{lemma}\label{t:summation_2_Bn}
Fix $i,j \ge 1$ and $\ve,\theta \in \ZZ_2$.
For any integer $m \ge 0$, let $x = x_{i,\ve}(m) \in B_{im}$ be a product of $m$ disjoint cycles, each of length $i$ and $\ZZ_2$-class $\ve$.
Let $R_{i,\ve,j,\theta}(m)$ be the set of all elements $z \in Z_{B_{im}}(x_{i,\ve}(m)) \cong G_{i,\ve} \wr S_m$ which, as elements of $B_{im}$, are products of $im/j$ disjoint cycles, each of length $j$ and $\ZZ_2$-class $\theta$.
(Of course, necessarily $j$ divides $im$.)
Then, in $\CC[[s]]$,
\[
\sum_{m \ge 0} 
\sum_{z \in R_{i,\ve,j,\theta}(m)} \omega^x(z)  
\frac{s^m}{m!}
= \exp \left( \sum_{e \in E(i,j)}
K_{\ve,\theta}(e) \,
\frac{(2is)^{j/e}}{2ij/e} \right) ,
\]
where $K_{\ve,\theta}(e)$ is defined as in Lemma~\ref{t:summation_1_Bn}.
\end{lemma}

\begin{proof}    
Assume that $E(i,j) = \{e_1, \ldots, e_q\}$, and define $\ell_k := j/e_k$ $(1 \le k \le q)$.
For each $m \ge 0$, let 
\[
N_{i,j}(m) 
:= \{(n_1,\ldots,n_q) \in \ZZ_{\ge 0}^q \,:\, n_1 \ell_1 + \ldots n_q \ell_q = m\}.  
\]
By Corollary~\ref{t:cycles2_Bn}, 
the possible cycle lengths of elements of $R_{i,\ve,j,\theta}(m)$, viewed as elements of $G_{i,\ve} \wr S_m$, are $\ell_1, \ldots, \ell_q$. If such an element has $n_k$ cycles of length $\ell_k$ $(1 \le k \le q)$, then clearly $(n_1,\ldots,n_q) \in N_{i,j}(m)$. 
The number of permutations in $S_m$ with this cycle structure is 
\[
\frac{m!}{n_1! \cdots n_q! \cdot \ell_1^{n_1} \cdots \ell_q^{n_q}}.
\]
By Lemma~\ref{t:summation_1_Bn}, for each common divisor $e$ of $i$ and $j$ and each specific cycle $\sigma \in S_m$ of length $\ell = j/e$,
\[
\sum_{z \in R_{i,\ve,j,\theta}(\sigma,e)} \omega^x(z) 
= (2i)^{\ell - 1} \cdot K_{\ve,\theta}(e).
\]    
The linearity of the character $\omega^x$ thus implies that
\begin{align*}
\sum_{z \in R_{i,\ve,j,\theta}(m)} \omega^x(z) 
&= \sum_{(n_1,\ldots,n_q) \in N_{i,j}(m)}
\frac{m!}{\prod_{k=1}^{q} n_k! \ell_k^{n_k}}
\cdot \prod_{k=1}^{q} 
\left( (2i)^{\ell_k - 1}  K_{\ve,\theta}(e_k) \right)^{n_k} \\
&= m! \sum_{(n_1,\ldots,n_q) \in N_{i,j}(m)} 
\prod_{k=1}^{q} \frac{1}{n_k!} \left( 
K_{\ve,\theta}(e_k) \,
\frac{(2i)^{\ell_k}}{2i \ell_k} \right)^{n_k} .
\end{align*}    
Let $s$ be an indeterminate, and consider the ring $\CC[[s]]$ of formal power series in $s$ over the field $\CC$.
By the definition of $N_{i,j}(m)$ and the above computation, it follows that the number
\[
\frac{1}{m!} \sum_{z \in R_{i,\ve,j,\theta}(m)} \omega^x(z) 
\]
is the coefficient of $s^m$ in the product
\begin{align*}
\prod_{k=1}^{q} \sum_{n_k=0}^{\infty} \frac{1}{n_k!} 
\left( K_{\ve,\theta}(e_k) \,
\frac{(2is)^{\ell_k}}{2i \ell_k} \right)^{n_k}
&= \prod_{k=1}^{q} 
\exp \left( K_{\ve,\theta}(e_k) \,
\frac{(2is)^{\ell_k}}{2i \ell_k} \right) \\
&= \prod_{e \in E(i,j)} 
\exp \left( K_{\ve,\theta}(e) \,
\frac{(2is)^{j/e}}{2ij/e} \right) .
\end{align*}
In other words,
\begin{align*}
\sum_{m \ge 0} 
\sum_{z \in R_{i,\ve,j,\theta}(m)} \omega^x(z) \,
\frac{s^m}{m!} 
&= \prod_{e \in E(i,j)}  
\exp \left( K_{\ve,\theta}(e) \,
\frac{(2is)^{j/e}}{2ij/e} \right) \\ 
&= \exp \left( \sum_{e \in E(i,j)}  
K_{\ve,\theta}(e) \,
\frac{(2is)^{j/e}}{2ij/e} \right) . 
\qedhere
\end{align*}
\end{proof}

Now let $s$ be an indeterminate and $\{t_{j,\theta} \,:\, j \ge 1, \theta \in \ZZ_2\}$ be a countable set of indeterminates, denoted succinctly by $\bt$. Consider the ring of formal power series $\CC[[s,\bt]]$.
We extend Lemma~\ref{t:summation_2_Bn} and compute the sum in Lemma~\ref{t:HLC_values} on the whole set $C_g \cap Z_x$, restricting only $x$ to have a cycle type with all cycles of the same length and $\ZZ_2$-class.

\begin{lemma}\label{t:summation_3_Bn}
Fix $i \ge 1$ and $\ve \in \ZZ_2$.
For any integer $m \ge 0$, let $x = x_{i,\ve}(m) \in B_{im}$ be a product of $m$ disjoint cycles, each of length $i$ and $\ZZ_2$-class $\ve$. Let $R_{i,\ve}(m) := Z_{B_{im}(x_{i,\ve}(m))} \cong G_{i,\ve} \wr S_m$. 
As an element of $B_{im}$, write each $z \in R_{i,\ve}(m)$ as a product of $m_{j,\theta}(z)$ disjoint cycles of length $j$ and $\ZZ_2$-class $\theta$ $(j \ge 1, \theta \in \ZZ_2)$.
Then, in $\CC[[s,\bt]]$,
\[
\sum_{m \ge 0} 
\sum_{z \in R_{i,\ve}(m)} 
\omega^x(z) \, 
\frac{s^m}{m!}
\prod_{j,\theta} t_{j,\theta}^{j m_{j,\theta}(z)}
= \exp \left( 
\sum_{j,\theta} \sum_{e \in E(i,j)}  
K_{\ve,\theta}(e) \,
\frac{(2is t_{j,\theta}^i)^{j/e}}{2ij/e} 
\right) .
\]
\end{lemma}

\begin{proof}
Following Lemma~\ref{t:summation_2_Bn}, fix integers $m_{j,\theta} \ge 0$ $(j \ge 1, \theta \in \ZZ_2)$ such that 
$\sum_{j,\theta} m_{j,\theta} = m$. 
Dividing the set of $m$ cycles of $x_{i,\ve}(m)$ into subsets of corresponding sizes $m_{j,\theta}$ can be done in
\[
\frac{m!}{\prod_{j,\theta} m_{j,\theta}!}
\]
ways. On each piece $G_{i,\ve} \wr S_{m_{j,\theta}}$ we would like to consider $R_{i,\ve,j,\theta}(m_{j,\theta})$, as in Lemma~\ref{t:summation_2_Bn}; note that, by that result, $R_{i,\ve,j,\theta}(m_{j,\theta}) = \varnothing$ unless $j$ divides $i m_{j,\theta}$.
By the linearity of $\omega^x$,
\[
\sum_{z \in R_{i,\ve}(m)} \omega^x(z)
= \sum_{\substack{m_{j,\theta} \ge 0 \\ \sum\limits_{j,\theta} m_{j,\theta} = m}}
\frac{m!}{\prod_{j,\theta} m_{j,\theta}!}
\prod_{j,\theta} 
\sum_{z_{j,\theta} \in R_{i,\ve,j,\theta}(m_{j,\theta})} \omega^x(z_{j,\theta}) 
\]
or, equivalently,
\[ 
\sum_{m \ge 0} \sum_{z \in R_{i,\ve}(m)} \omega^x(z) \,
\frac{s^m}{m!} 
= \prod_{j,\theta}         
\sum_{m_{j,\theta} \ge 0}
\sum_{z_{j,\theta} \in R_{i,\ve,j,\theta}(m_{j,\theta})} \omega^x(z_{j,\theta}) \,
\frac{s^{m_{j,\theta}}}{m_{j,\theta}!} .
\]
Assume now that $z \in R_{i,\ve}(m)$, as an element of $B_{im}$, is a product of $m_{j,\theta}(z)$ disjoint cycles of length $j$ and $\ZZ_2$-class $\theta$ $(j \ge 1, \theta \in \ZZ_2)$.
This yields a subdivision of the set of $m$ cycles (of length $i$ each) of $x_{i,\ve}(m)$ into subsets of sizes $m_{j,\theta} = j m_{j,\theta}(z)/i$, so that
\[
\sum_{j,\theta} j m_{j,\theta}(z) 
= \sum_{j,\theta} i m_{j,\theta} 
= i m.
\]
In order to keep track of the individual numbers $j m_{j,\theta}(z) = i m_{j,\theta}$, let us use additional indeterminates $t_{j,\theta}$ $(j \ge 1, \theta \in \ZZ_2)$. The previous formula turns into
\[ 
\sum_{m \ge 0} \sum_{z \in R_{i,\ve}(m)} \omega^x(z) \,
\frac{s^m}{m!} 
\prod_{j,\theta} t_{j,\theta}^{j m_{j,\theta}(z)}
= \prod_{j,\theta}         
\sum_{m_{j,\theta} \ge 0}
\sum_{z_{j,\theta} \in R_{i,\ve,j,\theta}(m_{j,\theta})} \omega^x(z_{j,\theta}) \,
\frac{s^{m_{j,\theta}} t_{j,\theta}^{i m_{j,\theta}}}{m_{j,\theta}!} .
\]
Using Lemma~\ref{t:summation_2_Bn}, with $s$ replaced by $s t_{j,\theta}^i$, this may be rewritten as 
\begin{align*}
\sum_{m \ge 0} 
\sum_{z \in R_{i,\ve}(m)} 
\omega^x(z) \,
\frac{s^m}{m!}
\prod_{j,\theta} t_{j,\theta}^{j m_{j,\theta}(z)} 
&= \prod_{j,\theta} \exp \left(
\sum_{e \in E(i,j)}  
K_{\ve,\theta}(e) \,
\frac{(2is t_{j,\theta}^i)^{j/e}}{2ij/e} 
\right) \\
&= \exp \left( 
\sum_{j,\theta} \sum_{e \in E(i,j)} 
K_{\ve,\theta}(e) \,
\frac{(2is t_{j,\theta}^i)^{j/e}}{2ij/e} 
\right) ,
\end{align*}
as claimed.
\end{proof}

Finally, let $s_{i,\ve}$ $(i \ge 1, \ve \in \ZZ_2)$ and $t_{j,\theta}$ $(j \ge 1, \theta \in \ZZ_2)$ be two countable sets of indeterminates, denoted succinctly by $\bs$ and $\bt$.
Consider the ring of formal power series $\CC[[\bs,\bt]]$.
For $\blambda \vdash n$ with $a_{i,\ve}$ parts of length $i$ and class $\ve$ $(i \ge 1, \ve \in \ZZ_2)$, $(i a_{i,\ve})_{i,\ve}$ is a decomposition of $n$. Denote 
\[
\bs^{\bc(\blambda)} := \prod_{i,\ve} s_{i,\ve}^{i a_{i,\ve}} .
\]
Similarly, for $\bmu \vdash n$ with $b_{j,\theta}$ parts of length $j$ and class $\theta$ $(j \ge 1, \theta \in \ZZ_2)$, denote 
\[
\bt^{\bc(\bmu)} := \prod_{j,\theta} t_{j,\theta}^{j b_{j,\theta}} .
\]
Recalling the higher Lie characters $\psi_{B_n}^\blambda$ from Definition~\ref{def:HLC_Bn}(c), we can now prove the main result of this section, Theorem~\ref{t:summation_4_Bn}.
It extends Lemma~\ref{t:summation_3_Bn} and computes the sum in Lemma~\ref{t:HLC_values} on the whole set $C_g \cap Z_x$, for arbitrary $g$ and $x$.

\begin{proof}[Proof of Theorem~\ref{t:summation_4_Bn}]
Write any $x \in B_n$ as 
\[
x = \prod_{i,\ve} x_{i,\ve} \,,
\]
where each $x_{i,\ve}$ is a product of $a_{i,\ve}$ disjoint cycles of length $i$ and $\ZZ_2$-class $\ve$.
Then, by Lemma~\ref{t:centralizer_in_Bn},
\[
Z_x \cong \bigtimes_{i,\ve} Z_{x_{i,\ve}}
\]
where
\[
Z_{x_{i,\ve}} \cong G_{i,\ve} \wr S_{a_{i,\ve}}.
\]
By Lemma~\ref{t:summation_3_Bn}, with summation over $m = a_{i,\ve} \ge 0$ and $z = z_{i,\ve} \in G_{i,\ve} \wr S_{a_{i,\ve}}$, we have for each $i \ge 1$ and $\ve \in \ZZ_2$:
\[
\sum_{a_{i,\ve} \ge 0} 
\sum_{z_{i,\ve} \in G_{i,\ve} \wr S_{a_{i,\ve}}} \omega^{x_{i,\ve}}(z_{i,\ve}) \,
\frac{s^{a_{i,\ve}}}{a_{i,\ve}!}
\prod_{j,\theta} t_{j,\theta}^{j m_{j,\theta}(z_{i,\ve})}
= \exp \left( 
\sum_{j,\theta} \sum_{e \in E(i,j)}  
K_{\ve,\theta}(e) \,
\frac{(2is t_{j,\theta}^i)^{j/e}}{2ij/e} 
\right) . 
\]
Replace $s$ by $s_{i,\ve}^i/2i$ and denote
\begin{align*}
\Sigma_{i,\ve}
&:= \sum_{a_{i,\ve} \ge 0} 
\sum_{z_{i,\ve} \in G_{i,\ve} \wr S_{a_{i,\ve}}} \omega^{x_{i,\ve}} (z_{i,\ve}) \,
\frac{s_{i,\ve}^{i a_{i,\ve}}}{a_{i,\ve}! \, (2i)^{a_{i,\ve}}}
\prod_{j,\theta} t_{j,\theta}^{j m_{j,\theta}(z_{i,\ve})} \\
&= \exp \left( 
\sum_{j,\theta} \sum_{e \in E(i,j)} 
K_{\ve,\theta}(e) \,
\frac{(s_{i,\ve} t_{j,\theta})^{ij/e}}{2ij/e} 
\right) . 
\end{align*}
The product of $\Sigma_{i,\ve}$ over all $i$ and $\ve$ is therefore
\begin{align*}
&\ \prod_{i,\ve} \sum_{a_{i,\ve} \ge 0} 
\sum_{z_{i,\ve} \in Z_{x_{i,\ve}}} \omega^{x_{i,\ve}}(z_{i,\ve}) \, 
\frac{s_{i,\ve}^{i a_{i,\ve}}}{a_{i,\ve}! \, (2i)^{a_{i,\ve}}}
\prod_{j,\theta} t_{j,\theta}^{j m_{j,\theta}(z_{i,\ve})} \\
&= \prod_{i,\ve} \Sigma_{i,\ve} 
= \exp \left( 
\sum_{i,\ve} \sum_{j,\theta} \sum_{e \in E(i,j)}  
K_{\ve,\theta}(e) \,
\frac{(s_{i,\ve} t_{j,\theta})^{ij/e}}{2ij/e} 
\right) .
\end{align*}
If $x = \prod_{i,\ve} x_{i,\ve} \in C_\blambda$ then
\[
|C_\blambda| 
= |C_x|
= \frac{|B_n|}{|Z_x|} 
= \frac{|B_n|}{\prod_{i,\ve} a_{i,\ve}!\, (2i)^{a_{i,\ve}}} .
\]
The above equality can thus be written as
\begin{align*}
&\ \sum_{n \ge 0} \frac{1}{|B_n|} 
\sum_{\blambda \vdash n} \sum_{x \in C_\blambda} 
\prod_{i,\ve} \left( 
\sum_{z_{i,\ve} \in Z_{x_{i,\ve}}} 
\omega^{x_{i,\ve}}(z_{i,\ve})  
s_{i,\ve}^{i a_{i,\ve}} \prod_{j,\theta} t_{j,\theta}^{j m_{j,\theta}(z_{i,\ve})} 
\right) \\
&= \exp \left( 
\sum_{i,\ve} \sum_{j,\theta} \sum_{e \in E(i,j)}  
K_{\ve,\theta}(e) \,
\frac{(s_{i,\ve} t_{j,\theta})^{ij/e}}{2ij/e} 
\right) .
\end{align*}
Denote $z := \prod_{i,\ve} z_{i,\ve} \in Z_x$, and note that
\[
\sum_{i,\ve} m_{j,\theta}(z_{i,\ve}) = b_{j,\theta}(z)
\qquad (\forall j \ge 1,\, \theta \in \ZZ_2) .
\]
By the definition of $\omega^x$, the LHS of the equality can thus be written as
\[
\text{LHS} 
= \sum_{n \ge 0} \frac{1}{|B_n|} 
\sum_{\blambda \vdash n} \sum_{x \in C_\blambda} \sum_{z \in Z_x} 
\omega^x(z) \bs^{\bc(x)} \bt^{\bc(z)} ,
\]
where $\bc(x) := (i a_{i,\ve})_{i,\ve}$ is the decomposition of $n$ corresponding to the cycle type of $x$, $\bc(z) := (j b_{j,\theta})_{j,\theta}$ is the decomposition corresponding to the cycle type of $z$, and
\[
\bs^{\bc(x)} := \prod_{i,\ve} s_{i,\ve}^{i a_{i,\ve}}, \quad
\bt^{\bc(z)} := \prod_{j,\theta} t_{j,\theta}^{j b_{j,\theta}} .
\]
In fact, we can rewrite this as
\[
\text{LHS} 
= \sum_{n \ge 0} \frac{1}{|B_n|} 
\sum_{\blambda \vdash n} \sum_{\bmu \vdash n} \sum_{x \in C_\blambda} \sum_{z \in Z_x \cap C_\bmu} 
\omega^x(z) 
\bs^{\bc(\blambda)} \bt^{\bc(\bmu)} ,
\]
since $\bc(x)$ depends only on the conjugacy class of $x$, and may thus be written as $\bc(\blambda)$; and similarly for $\bc(z)$ and $\bc(\bmu)$.

Now, by Lemma~\ref{t:HLC_values}, if $x \in C_\blambda$ then
\[
|C_\blambda| \cdot \sum_{z \in Z_x \cap C_\bmu} \omega^x(z) 
= |C_\bmu| \cdot \psi_{B_n}^x(\bmu) .
\]
Therefore
\begin{align*}
\text{LHS} 
&= \sum_{n \ge 0} \frac{1}{|B_n|} 
\sum_{\blambda \vdash n} \sum_{\bmu \vdash n} 
|C_\bmu| \, \psi_{B_n}^\blambda(\bmu) \,
\bs^{\bc(\blambda)} \bt^{\bc(\bmu)} \\
&= \sum_{n \ge 0}  
\sum_{\blambda \vdash n} \sum_{\bmu \vdash n} 
\psi_{B_n}^\blambda(\bmu) \,
\frac{\bs^{\bc(\blambda)} \bt^{\bc(\bmu)}}{|Z_\bmu|} .
\end{align*}
Regarding the RHS, note that $E(i,j)$ is the set of all common divisors of $i$ and $j$, namely divisors of $\gcd(i,j)$. We can therefore write the above equality as 
\begin{align*}
\sum_{n \ge 0}  
\sum_{\blambda \vdash n} \sum_{\bmu \vdash n} 
\psi_{B_n}^\blambda(\bmu) \,
\frac{\bs^{\bc(\blambda)} \bt^{\bc(\bmu)}}{|Z_\bmu|} 
&= \exp \left( 
\sum_{i,\ve} \sum_{j,\theta} \sum_{e | \gcd(i,j)}  
K_{\ve,\theta}(e) \,
\frac{(s_{i,\ve} t_{j,\theta})^{ij/e}}{2ij/e} 
\right) .
\end{align*}
This completes the proof.
\end{proof}

\section{Proof of Theorem~\ref{t:Bn-higher}}
\label{sec:proof_B}

In this section we use Theorem~\ref{t:gf_root_enumerator_Bn} and Theorem~\ref{t:summation_4_Bn} to prove Theorem~\ref{t:Bn-higher}.
First, some useful notations.

\begin{definition}\label{def:vdash_k} ($\blambda \vdash_k n$) \\
Let $k$ be an integer and $n$ a nonnegative integer. 
Assume that $\blambda = (\lambda^+,\lambda^-) \vdash n$ is a bipartition, with $a_i^\ve$ parts of size $i$ and type $\ve$ $(\forall\, i \ge 1,\, \ve \in \ZZ_2)$.
We write $\blambda \vdash_k n$ if one of the following conditions holds. 
Note that, by Lemma~\ref{t:power_of_one_cycle_Bn} with $j = 1$ and $\theta = +1 \in \ZZ_2$, these conditions are equivalent. 
\begin{enumerate}

\item[(a)]
All (equivalently, one of) the elements $x \in B_n$ of cycle type $\blambda$ satisfy $x^k = 1_{B_n}$.

\item[(b)]
$a^\ve_i = 0$, unless $i$ divides $k$ and $\ve^{k/i} = 1$ in $\ZZ_2$.

\item[(c)]
All part sizes in $\lambda^+$ divide $k$, and all part sizes in $\lambda^-$ divide $k/2$. 
\end{enumerate}
\end{definition}

\begin{definition}\label{def:psi_vdash_k}
Let $k$ be an integer and $n$ a nonnegative integer. Denote
\[
\psi_k^{B_n} := \sum_{\blambda \vdash_k n} \psi_{B_n}^\blambda .
\]      
\end{definition}

Theorem~\ref{t:Bn-higher} can now be restated as the following claim, which we shall prove.

\begin{theorem}\label{t:roots_eq_psi}
For any integer $k$ and nonnegative integer $n$,
\[
\roots_k^{B_n} = \psi_k^{B_n} .
\]
\end{theorem}

\begin{proof}
By Theorem~\ref{t:gf_root_enumerator_Bn},
\[
\sum_{n \ge 0} \sum_{\bmu \vdash n} 
\roots_k^{B_n}(\bmu) \,
\frac{\bt^{\bc(\bmu)}}{|Z_\bmu|}
= \exp \left(
\sum_{j,\theta}
\sum_{\substack{h|k \\ \gcd(h,j) = 1}} 
\roots_h^{\ZZ_2}(\theta) \,
\frac{t_{j,\theta}^{j k/h}}{2 j k/h}         \right).
\]
On the other hand, by Theorem~\ref{t:summation_4_Bn},
\[
\sum_{n \ge 0}  
\sum_{\blambda \vdash n} \sum_{\bmu \vdash n} 
\psi_{B_n}^\blambda(\bmu) \,
\frac{\bs^{\bc(\blambda)} \bt^{\bc(\bmu)}}{|Z_\bmu|}
= \exp \left( 
\sum_{i,\ve} \sum_{j,\theta} \sum_{e | \gcd(i,j)}  
K_{\ve,\theta}(e) \,
\frac{(s_{i,\ve} t_{j,\theta})^{ij/e}}{2ij/e} 
\right) ,
\]
where 
\[
K_{\ve,\theta}(e)
:= \ve \theta \cdot \mu(2e)
+ \frac{(1+\ve)(1+\theta)}{2} \cdot \mu(e).
\] 
Letting
\[
s_{i,\ve} :=
\begin{cases}
1, &\text{if } i|k \text{ and } \ve^{k/i} = 1; \\
0, &\text{otherwise,}
\end{cases}
\]
we obtain, by Definition~\ref{def:psi_vdash_k},
\begin{align*}
\sum_{n \ge 0} \sum_{\bmu \vdash n} 
\psi_k^{B_n}(\bmu) \,
\frac{\bt^{\bc(\bmu)}}{|Z_\bmu|}
&= 
\sum_{n \ge 0} \sum_{\bmu \vdash n} 
\sum_{\blambda \vdash_k n} 
\psi_{B_n}^\blambda(\bmu) \,
\frac{\bt^{\bc(\bmu)}}{|Z_\bmu|} \\
&= 
\exp \left( \sum_{j,\theta} 
\sum_{\substack{i \\ i|k}} 
\sum_{\substack{\ve \\ \ve^{k/i} = 1}} 
\sum_{\substack{e \\ e | \gcd(i,j)}}  
K_{\ve,\theta}(e) \,
\frac{t_{j,\theta}^{ij/e}}{2ij/e} \right) .
\end{align*}
Fixing the integer $k$, it follows that the claim 
\[
\roots_k^{B_n} = \psi_k^{B_n}
\qquad (\forall n \ge 0)
\]
is equivalent to the claim that, for any $j \ge 1$ and $\theta \in \ZZ_2$,
\begin{equation}\label{eq:roots_hLc_1_Bn}
\sum_{\substack{h \\ h|k \\ \gcd(h,j) = 1}} 
\roots_h^{\ZZ_2}(\theta) \, 
\frac{t_{j,\theta}^{j k/h}}{2 j k/h}         = \sum_{\substack{i \\ i|k}} 
\sum_{\substack{\ve \\ \ve^{k/i} = 1}} 
\sum_{\substack{e \\ e | \gcd(i,j)}} 
K_{\ve,\theta}(e) \,
\frac{t_{j,\theta}^{ij/e}}{2ij/e} .
\end{equation}
Letting $d := i/e$ and then $h := k/d$ on the RHS of~\eqref{eq:roots_hLc_1_Bn} yields
\begin{align*}
\text{RHS}
&= \sum_{\substack{i \\ i|k}} 
\sum_{\substack{e \\ e|i \\ e|j}}  
\sum_{\substack{\ve \\ \ve^{k/i} = 1}}
K_{\ve,\theta}(e) \,
\frac{t_{j,\theta}^{ij/e}}{2ij/e} \\
&= \sum_{\substack{d \\ d|k}} 
\sum_{\substack{e \\ e|(k/d) \\ e|j}}  
\sum_{\substack{\ve \\ \ve^{k/(de)} = 1}} 
K_{\ve,\theta}(e) \,
\frac{t_{j,\theta}^{j d}}{2j d} \\
&= \sum_{\substack{h \\ h|k}} 
\sum_{\substack{e \\ e|h \\ e|j}}  
\sum_{\substack{\ve \\ \ve^{h/e} = 1}} 
K_{\ve,\theta}(e) \,
\frac{t_{j,\theta}^{j k/h}}{2j k/h} .
\end{align*}
Comparing powers of $t_{j,\theta}$ in~\eqref{eq:roots_hLc_1_Bn}, it follows that we need to prove that, for any $j \ge 1$, $\theta \in \ZZ_2$ and a divisor $h$ of $k$,
\begin{equation}\label{eq:roots_hLc_2_Bn}
\sum_{\substack{e \\ e|\gcd(h,j)}} 
\sum_{\substack{\ve \\ \ve^{h/e} = 1}} 
K_{\ve,\theta}(e) 
= \begin{cases}
	\roots_h^{\ZZ_2}(\theta), &\text{if } \gcd(h,j) = 1; \\
	0, &\text{otherwise.}
\end{cases}
\end{equation}
Indeed, by 
the explicit formula for $K_{\ve,\theta}(e)$,
\[
\sum_{\substack{\ve \\ \ve^{h/e} = 1}} 
K_{\ve,\theta}(e) 
= \begin{cases}
\theta \cdot \mu(2e) + (1 + \theta) \cdot \mu(e), &\text{if } h/e \text{ is odd;} \\
(1 + \theta) \cdot \mu(e), &\text{if } h/e \text{ is even.}
\end{cases}
\]
By a well-known property of the M\"obius function,
\[
\sum_{\substack{e \\ e|\gcd(h,j)}} \mu(e)
= \begin{cases}
	1, &\text{if } \gcd(h,j) = 1; \\
	0, &\text{otherwise.} 
\end{cases}
\]
Since 
\[
\mu(2e)
= \begin{cases}
	-\mu(e), &\text{if $e$ is odd;} \\
	0, &\text{if $e$ is even,} 
\end{cases}
\]
it follows that
\begin{align*}
	\sum_{\substack{e \\ e|\gcd(h,j) \\ h/e \text{ odd}}} \mu(2e)
	&= -\sum_{\substack{e \text{ odd} \\ e|\gcd(h,j) \\ h/e \text{ odd}}} \mu(e)
	= \begin{cases}
		-\sum_{e|\gcd(h,j)} \mu(e), &\text{if $h$ is odd;} \\
		0, &\text{if $h$ is even} 
	\end{cases} \\
	&= \begin{cases}
		-1, &\text{if } \gcd(h,j) = 1 \text{ and } h \text{ is odd;} \\
		0, &\text{otherwise.} 
	\end{cases}
\end{align*}
Thus
\[
\sum_{\substack{e \\ e|\gcd(h,j)}} 
\sum_{\substack{\ve \\ \ve^{h/e} = 1}} 
K_{\ve,\theta}(e) 
= \begin{cases}
	1, &\text{if } \gcd(h,j) = 1 \text{ and } h \text{ is odd;} \\
	1 + \theta, &\text{if } \gcd(h,j) = 1 \text{ and } h \text{ is even;} \\
	0, &\text{otherwise.}
\end{cases}
	\]
	If $\gcd(h,j) \ne 1$ then the RHS here is zero, and if $\gcd(h,j) = 1$ then the RHS is $\roots_h^{\ZZ_2}(\theta)$.
	This proves~\eqref{eq:roots_hLc_2_Bn} and completes the proof of the theorem.
\end{proof}

\section{Root enumerators of type $D$}
\label{sec:D_proper}

In this section we prove   
Theorems~\ref{t:Dn-proper}, \ref{t:Dn-higher} and~\ref{t:Dn-higher_odd}. 
In Subsection~\ref{subsec:sroots} we 
define signed root enumerators of type $B$ and state the main result of this section, 
Theorem~\ref{t:sroots_eq_spsi}, which is a signed analogue  of Theorem~\ref{t:roots_eq_psi}.
Its proof uses Theorem~\ref{t:gf_sroot_enumerator_Bn}, which provides a generating function for signed root enumerators in $B_n$.
Theorem~\ref{t:gf_sroot_enumerator_Bn} is stated and proved in Subsection~\ref{subsec:gf_signed_roots},
and is used in Subsection~\ref{subsec:sroots_and_hLc}, together with Theorem~\ref{t:summation_4_Bn}, 
to prove Theorem~\ref{t:sroots_eq_spsi}. 
Then, in Subsection~\ref{subsec:D_proper}, we deduce Theorems~\ref{t:Dn-higher} and~\ref{t:Dn-proper}.
Finally, in Subsection~\ref{subsec:HLC_Dn}, we use Theorem~\ref{t:Dn-higher} to prove Theorem~\ref{t:Dn-higher_odd}. 

\subsection{Signed root enumerators of type $B$}
\label{subsec:sroots}

Let us first recall the linear character $\chi$ on $B_n$.

\begin{definition}\label{def:sign_char_Bn}
	Let $\chi$ be the linear character on $B_n = \ZZ_2 \wr S_n$ defined as follows:
	$\chi$ is the identity map (i.e., the sign character) on $\ZZ_2 = \{+1,-1\}$, and is trivial on the wreathing group $S_n$.
\end{definition}

Equivalently, for an element $x \in B_n$, $\chi(x)$ is the product of the $\ZZ_2$-classes of the cycles of $x$, which is also the product of the signs of all the entries in the window notation $[x(1),\ldots,x(n)]$.

\begin{remark}\label{rem:D} 
	In the sequel we will consider the Weyl group $D_n$ as a subgroup of $B_n$, as follows. 
	\[
	D_n 
	= \ker(\chi) 
	= \{x \in B_n \,:\, \chi(x) = 1\}. 
	\]
\end{remark}


Recall now that, for an integer $k$ and a nonnegative integer $n$, the $k$-th root enumerator in $B_n$ is defined by
\[
\roots_k^{B_n}(y) 
:= |\{x \in B_n \,:\, x^k = y\}| 
= \sum_{\substack{x \in B_n \\ x^k = y}} 1 
\qquad (y \in B_n) .
\]

\begin{definition}\label{def:sroots_Bn}
	Let $k$ be an integer and $n$ a nonnegative integer.
	The {\em signed $k$-th root enumerator} in $B_n$ is defined by 
	\[
	\sroots_k^{B_n}(y) 
	:= \sum_{\substack{x \in B_n \\ x^k = y}} \chi(x)
	\qquad (y \in B_n) .
	\]
	We shall also use the signed $k$-th root enumerator in $\ZZ_2 = \{+1,-1\}$,
	\[
	\sroots_k^{\ZZ_2}(\theta) 
	:= \sum_{\substack{\ve \in \ZZ_2 \\ \ve^k = \theta}} \chi(\ve)
	= \sum_{\substack{\ve \in \ZZ_2 \\ \ve^k = \theta}} \ve
	\qquad (\theta \in \ZZ_2) .
	\]    
\end{definition}

In order to state our main result, Theorem~\ref{t:sroots_eq_spsi}, we need two more definitions.

\begin{definition}\label{def:vdash_sk} ($\blambda \vdash_\sk n$) \\
	Let $k$ be an integer and $n$ a nonnegative integer. 
	Assume that $\blambda = (\lambda^+,\lambda^-) \vdash n$ is a bipartition, with $a_i^\ve$ parts of size $i$ and type $\ve$ $(\forall\, i \ge 1,\, \ve \in \ZZ_2)$.
	Let $w_0 := [-1,-2,\ldots,-n]$, the unique element of $B_n$ with all cycles of length $1$ and $\ZZ_2$-class $-1$.
	This is a central involution and the longest element in $B_n$.
	We write $\blambda \vdash_\sk n$ if one of the following conditions holds. 
	Note that, by Lemma~\ref{t:power_of_one_cycle_Bn} with $j = 1$ and $\theta = -1 \in \ZZ_2$, these conditions are equivalent. 
	\begin{enumerate}
		
		\item[(a)]
		All (equivalently, one of) the elements $x \in B_n$ of cycle type $\blambda$ satisfy $x^k = w_0$.
		
		\item[(b)]
		$a^\ve_i = 0$, unless $i$ divides $k$ and $\ve^{k/i} = -1$ in $\ZZ_2$.
		
		\item[(c)]
		$\lambda^+ = \varnothing$, and all part sizes in $\lambda^-$ divide $k$ but do not divide $k/2$. 
	\end{enumerate}
\end{definition}

\begin{definition}\label{def:psi_vdash_sk}
	Let $k$ be an integer and $n$ a nonnegative integer. Denote
	\[
	\spsi_k^{B_n} := \sum_{\blambda \vdash_\sk n} \psi_{B_n}^\blambda .
	\]      
\end{definition}

The main result of this section is the following.

\begin{theorem}\label{t:sroots_eq_spsi}
	For any integer $k$ and nonnegative integer $n$,
	\[
	\sroots_k^{B_n} = \spsi_k^{B_n} .
	\]
\end{theorem}

In Subsection~\ref{subsec:gf_signed_roots} we compute a generating function (Theorem~\ref{t:gf_sroot_enumerator_Bn}) for signed root enumerators in $B_n$. 
In Subsection~\ref{subsec:sroots_and_hLc} we combine this result with Theorem~\ref{t:summation_4_Bn} to deduce Theorem~\ref{t:sroots_eq_spsi}.
In Subsection~\ref{subsec:D_proper} we use Theorems~\ref{t:roots_eq_psi} and~\ref{t:sroots_eq_spsi} to deduce Theorems~\ref{t:Dn-higher} and~\ref{t:Dn-proper}.
Finally, in Subsection~\ref{subsec:HLC_Dn}, we use Theorem~\ref{t:Dn-higher} to prove Theorem~\ref{t:Dn-higher_odd}.

\subsection{Generating function for signed root enumerators}
\label{subsec:gf_signed_roots}

In this subsection we state and prove a generating function (Theorem~\ref{t:gf_sroot_enumerator_Bn}) for the values of the signed $k$-th root enumerator $\sroots_k^{B_n}$ on arbitrary elements of $B_n$, in terms of signed $h$-th root enumerators $\sroots_h^{\ZZ_2}$ for various divisors $h$ of $k$.

Let $\bmu = (\mu^+,\mu^-)$ be a bipartition of $n$. For each integer $j \ge 1$ and sign $\theta \in \ZZ_2 = \{+1,-1\}$, let $b_{j,\theta}$ be the number of parts of size $j$ in the partition $\mu^\theta$. Thus
\[
\sum_{j,\theta} j b_{j,\theta} = n.
\]
Let $\bt = (t_{j,\theta})_{j \ge 1, \theta \in \ZZ_2}$ be a countable set of indeterminates, and denote $\bt^{\bc(\bmu)} := \prod_{j,\theta} t_{j,\theta}^{j b_{j,\theta}}$.
Consider the ring $\CC[[\bt]]$ of formal power series in these indeterminates.
The main result of this subsection is the following signed analogue of Theorem~\ref{t:gf_root_enumerator_Bn}.

\begin{theorem}\label{t:gf_sroot_enumerator_Bn}
	For any integer $k$, the generating function for the signed $k$-th root enumerator in $B_n$ is
	\[
	\sum_{n \ge 0} 
	\sum_{\bmu \vdash n} 
	\sroots_k^{B_n}(\bmu) \,
	\frac{\bt^{\bc(\bmu)}}{|Z_\bmu|}
	= \exp \left(
	\sum_{j,\theta}
	\sum_{\substack{h|k \\ \gcd(h,j) = 1}} 
	\sroots_h^{\ZZ_2}(\theta) \,
	\frac{t_{j,\theta}^{j k/h}}{2 j k/h}         \right) .
	\]
\end{theorem}

The rest of this subsection is devoted to the proof of Theorem~\ref{t:gf_sroot_enumerator_Bn}.

Consider two countable sets of indeterminates, $\br = (r_{\ell,\ve})_{\ell \ge 1,\ve \in \ZZ_2}$ and $\bt = (t_{j,\theta})_{j \ge 1,\theta \in \ZZ_2}$.
We shall work in the ring $\CC[[\br,\bt]]$ of formal power series in these indeterminates, over the field $\CC$ of complex numbers.
Our first result is a signed analogue of Lemma~\ref{t:root_enumerator_similar_cycles_Bn}.
It gives a generating function for the sum, refined by cycle type, of $\chi$-values of the $k$-th roots of an element of $B_n$ which is a product of disjoint cycles of a fixed type, namely length $j \ge 1$ and $\ZZ_2$-class $\theta$.

\begin{lemma}\label{t:sroot_enumerator_similar_cycles_Bn} 
	Fix $j \ge 1$ and $\theta \in \ZZ_2$.
	For each $b \ge 0$, pick an element $y \in B_{j b}$ which is a product of $b$ disjoint cycles, each of length $j$ and $\ZZ_2$-class $\theta$. 
	The generating function for the sum of $\chi$-values of elements $x \in B_{j b}$ satisfying $x^k = y$ is
	\begin{align*}
		\sum_{b \ge 0} \frac{t_{j,\theta}^{j b}}{b! (2j)^b} 
		\sum_{\blambda \vdash j b} \sum_{\substack{x \in C_\blambda \\ x^k = y}} \chi(x)
		\prod_{\ell,\ve} r_{\ell,\ve}^{\ell m_{\ell,\ve}}
		&= \exp \left( 
		\sum_{\substack{h | k \\ \gcd(h,j) = 1}} 
		\frac{t_{j,\theta}^{j k/h}}{2 j k/h} \sum_{\ve \in R_{h,\theta}} \ve r_{j k/h,\ve}^{j k/h} 
		\right) ,
	\end{align*}
	where $R_{h,\theta} := \{\ve \in \ZZ_2 \,:\, \ve^h = \theta\}$.
\end{lemma}

\begin{proof}
	As in the proof of Lemma~\ref{t:root_enumerator_similar_cycles_Bn}, denote
	\[
	D_k(j) := 
	\{d \,:\, d \ge 1, \,\, d|k, \,\, \gcd(k/d,j) = 1\} ,
	\]
	and assume that $d_1, \ldots, d_q$ are the distinct elements of $D_k(j)$.
	For each $1 \le i \le q$, let $m_i$ be the number of cycles of $x$ of length $\ell_i = d_i j$. 
	The total number of cycles of $y$, each of length $j$, is then $m_1 d_1 + \ldots + m_q d_q = b$.
	It follows that the cycle-length distributions of elements $x \in B_{j b}$ satisfying $x^k = y$ correspond to elements of the set
	\[
	M_k(j,b) := 
	\{(m_1, \ldots, m_q) \in \ZZ_{\ge 0}^q \,:\, m_1 d_1 + \ldots + m_q d_q = b\} .
	\]
	The $\ZZ_2$-class $\ve$ of a cycle of length $d_i j$ must satisfy $\ve^{k/d_i} = \theta$.
	For each integer $h$, let
	\[
	R_{h,\theta} 
	:= \{\ve \in \ZZ_2 \,:\, \ve^h = \theta\},
	\]
	so that $|R_{h,\theta}| = \roots^{\ZZ_2}_h(\theta)$.
	
	The size of the conjugacy class of $y$, as computed in the proof of Lemma~\ref{t:root_enumerator_similar_cycles_Bn}, is
	\[
	|C_y|
	= \frac{(j b)! \cdot (2^{j-1})^b}{b! \cdot j^b} 
	= \frac{(j b)! \cdot 2^{j b}}{b! \cdot (2j)^b} .
	\]
	Assume now that $x$ has $m_{i,\ve}$ cycles of length $\ell_i = d_i j$ and $\ZZ_2$ class $\ve$ $(1 \le i \le q,\, \ve \in R_{k/d_i,\theta})$,
	so that $\sum_\ve m_{i,\ve} = m_i$ $(1 \le i \le q)$.
	A similar computation shows that the size of the conjugacy class of $x$ is
	\[
	|C_x|
	= \frac{(j b)! \cdot 
		\prod_{i=1}^{q} \prod_{\ve \in R_{k/d_i,\Theta}} (2^{\ell_i - 1})^{m_{i,\ve}}} 
	{\prod_{i=1}^{q} \prod_{\ve \in R_{k/d_i,\theta}} m_{i,\ve}! \cdot \ell_i^{m_{i,\ve}}} 
	= (j b)! \cdot 
	\prod_{i=1}^{q} \left( \frac{2^{\ell_i}}{2 \ell_i} \right)^{m_i}
	\prod_{\ve \in R_{k/d_i,\theta}} \frac{1}{m_{i,\ve}!} .
	\]
	Dividing this by the size of the conjugacy class of $y$ (and using $\sum_i m_i d_i = b$ which is equivalent to $\sum_i m_i \ell_i = j b$) yields
	\[
	\frac{|C_x|}{|C_y|}
	= b! (2j)^b \cdot 
	\prod_{i=1}^{q} \frac{1}{(2 \ell_i)^{m_i}} 
	\prod_{\ve \in R_{k/d_i,\theta}} \frac{1}{m_{i,\ve}!} .
	\]
	This is the number of elements $x' \in C_x$ satisfying $(x')^k = y$.
	Each of these elements has
	\begin{equation}\label{eq:chi_1}
		\chi(x') = \chi(x) 
		= \prod_{i=1}^{q} 
		\prod_{\ve \in R_{k/d_i,\theta}}
		\ve^{m_{i,\ve}} .
	\end{equation}
	In order to record the cycle types of $x$ and $y$, let us introduce two countable sets of indeterminates, $(r_{\ell,\ve})_{\ell \ge 1,\ve \in \ZZ_2}$ and $(t_{j,\theta})_{j \ge 1,\theta \in \ZZ_2}$. Multiply the number of $x' \in C_x$ satisfying $(x')^k = y$ by $\chi(x')$ and by the monomial $\prod_{i,\ve} r_{\ell_i,\ve}^{\ell_i m_{i,\ve}}$ to get
	\begin{equation}\label{eq:chi_2}
		b! \left( 2j \right)^b \cdot 
		\prod_{i=1}^{q} \frac{1}{(2 \ell_i)^{m_i}} 
		\prod_{\ve \in R_{k/d_i,\theta}} \frac{(\ve r_{\ell_i,\ve}^{\ell_i})^{m_{i,\ve}}}{m_{i,\ve}!} .
	\end{equation}
	Sum this, for each $1 \le i \le q$, over all the decompositions of $m_i$ into a sum of nonnegative integers $m_{i,\ve}$ $(\ve \in R_{k/d_i,\theta})$, and use the multinomial identity
	\[
	\left( \sum_{\ve \in R} x_\ve \right)^m
	= m! \cdot 
	\sum_{\substack{m_\ve \ge 0 \, (\ve \in R) \\ \sum_{\ve} m_\ve = m}} \,\,
	\prod_{\ve \in R} \frac{x_\ve^{m_\ve}}{m_\ve!} 
	\]
	for $x_\ve = \ve r_{\ell_i,\ve}^{\ell_i}$, $m_\ve = m_{i,\ve}$, $m = m_i$ and $R = R_{k/d_i,\theta}$,
	to get
	\begin{align*}
		&\ b! \left( 2j \right)^b \cdot 
		\prod_{i=1}^{q} \frac{1}{(2 \ell_i)^{m_i}} 
		\sum_{\substack{m_{i,\ve} \ge 0 \, (\ve \in R_{k/d_i,\theta}) \\ \sum_{\ve} m_{i,\ve} = m_i}}
		\prod_{\ve \in R_{k/d_i,\theta}} \frac{(\ve r_{\ell_i,\ve}^{\ell_i})^{m_{i,\ve}}}{m_{i,\ve}!} \\
		&= b! \left( 2j \right)^b \cdot 
		\prod_{i=1}^{q} \frac{1}{(2 \ell_i)^{m_i} m_i!} 
		\left( \sum_{\ve \in R_{k/d_i,\theta}} \ve r_{\ell_i,\ve}^{\ell_i} \right)^{m_i} .
	\end{align*}
	Now sum this over all $(m_1,\ldots,m_q) \in M_k(j,b)$, namely $m_i \ge 0$ such that $m_1 \ell_1 + \ldots + m_q \ell_q = j b$, to get
	\[
	b! \left( 2j \right)^b \cdot 
	\sum_{(m_1,\ldots,m_q) \in M_k(j,b)} 
	\prod_{i=1}^{q} \frac{1}{m_i!} 
	\left( \frac{1}{2 \ell_i} \sum_{\ve \in R_{k/d_i,\theta}} \ve r_{\ell_i,\ve}^{\ell_i} \right)^{m_i} .
	\]
	Dividing by $b! (2j)^b$, the above sum is the coefficient of $t_{j,\theta}^{j b}$ in
	\[
	\prod_{i=1}^{q} \sum_{m_i \ge 0}
	\frac{1}{m_i!} 
	\left( \frac{t_{j,\theta}^{\ell_i}}{2 \ell_i} \sum_{\ve \in R_{k/d_i,\theta}} \ve r_{\ell_i,\ve}^{\ell_i} \right)^{m_i} \\
	= \prod_{i=1}^{q} 
	\exp \left( \frac{t_{j,\theta}^{\ell_i}}{2 \ell_i} \sum_{\ve \in R_{k/d_i,\theta}} \ve r_{\ell_i,\ve}^{\ell_i} \right) .
	\]
	By the above computations, this is (up to a factor $b! (2j)^b$) the generating function for $\chi(x)$ over all the elements $x \in B_{j b}$ satisfying $x^k = y$ for a specific element $y$, which has $b$ cycles of length $j$ and $\ZZ_2$-class $\theta$.
	The solutions are counted by conjugacy class: If $x$ has $m_{\ell,\ve}$ cycles of length $\ell$ and $\ZZ_2$-class $\ve$ $(\ell \ge 1, \ve \in \ZZ_2)$, then it contributes $\chi(x)$ to the coefficient of $t_{j,\theta}^{j b} \prod_{\ell,\ve} r_{\ell,\ve}^{\ell m_{\ell,\ve}}$.
	The generating function includes summation over $b \ge 0$. Therefore, recovering the factor $b! (2j)^b$,
	\[
	\sum_{b \ge 0} \frac{t_{j,\theta}^{j b}}{b! (2j)^b} 
	\sum_{\blambda \vdash j b} \sum_{\substack{x \in C_\blambda \\ x^k = y}} \chi(x)
	\prod_{\ell,\ve} r_{\ell,\ve}^{\ell m_{\ell,\ve}} 
	= \prod_{i=1}^{q} 
	\exp \left( 
	\frac{t_{j,\theta}^{\ell_i}}{2 \ell_i} \sum_{\ve \in R_{k/d_i,\theta}} \ve r_{\ell_i,\ve}^{\ell_i} 
	\right) .
	\]
	Recall now that $d_i$ $(1 \le i \le q)$ are all the divisors of $k$ which satisfy $\gcd(k/d_i,j) = 1$, so that $h_i := k/d_i$ are the divisors of $k$ which satisfy $\gcd(h_i,j) = 1$. Also, $\ell_i = d_i j = j k/h_i$. Thus
	\begin{align*}
		\sum_{b \ge 0} \frac{t_{j,\theta}^{j b}}{b! (2j)^b} 
		\sum_{\blambda \vdash j b} \sum_{\substack{x \in C_\blambda \\ x^k = y}} \chi(x)
		\prod_{\ell,\ve} r_{\ell,\ve}^{\ell m_{\ell,\ve}} 
		&= \prod_{\substack{h | k \\ \gcd(h,j) = 1}} 
		\exp \left( 
		\frac{t_{j,\theta}^{j k/h}}{2 j k/h} \sum_{\ve \in R_{h,\theta}} \ve r_{j k/h,\ve}^{j k/h} 
		\right) ,
	\end{align*}
	as claimed.
\end{proof} 

The general case follows.
It is a signed analogue of Theorem~\ref{t:gf_refined_root_enumerator_Bn}.

\begin{theorem}\label{t:gf_refined_sroot_enumerator_Bn}
	For any integer $k$, the refined generating function for $\chi(x)$ over all elements $x \in B_n$ satisfying $x^k = y$ is
	\begin{align*}
		&\ \sum_{n \ge 0} \frac{1}{|B_n|}
		\sum_{\bmu \vdash n} \sum_{y \in C_\bmu} 
		\sum_{\blambda \vdash n} 
		\sum_{\substack{x \in C_\blambda \\ x^k = y}} \chi(x)
		\prod_{\ell,\ve} r_{\ell,\ve}^{\ell m_{\ell,\ve}} 
		\prod_{j,\theta} t_{j,\theta}^{j b_{j,\theta}} \\
		&= \exp \left(
		\sum_{j,\theta}
		\sum_{\substack{h|k \\ \gcd(h,j) = 1}} 
		\frac{t_{j,\theta}^{j k/h}}{2 j k/h} \sum_{\ve \in R_{h,\theta}} \ve r_{j k/h,\ve}^{j k/h} 
		\right) ,
	\end{align*}
	where $R_{h,\theta} := \{\ve \in \ZZ_2 \,:\, \ve^h = \theta\}$, cycle-type $\blambda$ has $m_{\ell,\ve}$ cycles of length $\ell$ and $\ZZ_2$-class $\ve$, and cycle-type $\bmu$ has $b_{j,\theta}$ cycles of length $j$ and $\ZZ_2$-class $\theta$. 
\end{theorem}

\begin{proof}
	Let $y = \prod_{j,\theta} y_{j,\theta}$, where $y_{j,\theta}$ has $b_{j,\theta}$ cycles, each of length $j$ and $\ZZ_2$-class $\theta$. Use Lemma~\ref{t:sroot_enumerator_similar_cycles_Bn}, with $b = b_{j,\theta}$, and multiply over all $j$ and $\theta$ to get
	\begin{align*}
		&\ \prod_{j,\theta} 
		\sum_{b_{j,\theta} \ge 0} \frac{t_{j,\theta}^{j b_{j,\theta}}}{b_{j,\theta}! (2j)^{b_{j,\theta}}} 
		\sum_{\blambda \vdash j b_{j,\theta}} \sum_{\substack{x \in C_\blambda \\ x^k = y_{j,\theta}}} \chi(x)
		\prod_{\ell,\ve} r_{\ell,\ve}^{\ell m_{\ell,\ve}} \\
		&= \prod_{j,\theta}
		\prod_{\substack{h | k \\ \gcd(h,j) = 1}} 
		\exp \left( 
		\frac{t_{j,\theta}^{j k/h}}{2 j k/h} \sum_{\ve \in R_{h,\theta}} \ve r_{j k/h,\ve}^{j k/h} 
		\right) ,
	\end{align*}
	where $R_{h,\theta} := \{\ve \in \ZZ_2 \,:\, \ve^h = \theta\}$. 
	Using
	\[
	|C_y| = \frac{|B_n|}{\prod_{j,\theta} b_{j,\theta}! (2j)^{b_{j,\theta}}}
	\]
	gives the equivalent form
	\begin{align*}
		&\ \sum_{n \ge 0} \frac{1}{|B_n|}
		\sum_{\bmu \vdash n} \sum_{y \in C_\bmu} 
		\sum_{\blambda \vdash n} 
		\sum_{\substack{x \in C_\blambda \\ x^k = y}} \chi(x)
		\prod_{\ell,\ve} r_{\ell,\ve}^{\ell m_{\ell,\ve}} 
		\prod_{j,\theta} t_{j,\theta}^{j b_{j,\theta}} \\
		&= \prod_{j,\theta}
		\prod_{\substack{h | k \\ \gcd(h,j) = 1}} 
		\exp \left( 
		\frac{t_{j,\theta}^{j k/h}}{2 j k/h} \sum_{\ve \in R_{h,\theta}} \ve r_{j k/h,\ve}^{j k/h} 
		\right) ,
	\end{align*}
	as claimed.
\end{proof}

\begin{proof}[Proof of Theorem~\ref{t:gf_sroot_enumerator_Bn}]
	Set $r_{\ell,\ve} = 1$ in Theorem~\ref{t:gf_refined_sroot_enumerator_Bn}, for all $\ell$ and $\ve$, and note that 
	\[
	\sum_{\blambda \vdash n} 
	\sum_{\substack{x \in C_\blambda \\ x^k = y}} \chi(x)
	= \sroots_k^{B_n}(y) .
	\]
	Recall that $\sum_{\ve \in R_{h,\theta}} \ve = \sroots_h^{\ZZ_2}(\theta)$ and $|Z_\bmu| = |B_n|/|C_\bmu|$.
	Denoting $\bt^{\bc(\bmu)} :=         \prod_{j,\theta} t_{j,\theta}^{j b_{j,\theta}}$, Theorem~\ref{t:gf_refined_sroot_enumerator_Bn} reduces to Theorem~\ref{t:gf_sroot_enumerator_Bn}, presenting the (unrefined) generating function for the signed $k$-th root enumerator in $B_n$.
\end{proof}

\subsection{Signed root enumerators and higher Lie characters}
\label{subsec:sroots_and_hLc}

We now use Theorem~\ref{t:gf_sroot_enumerator_Bn} and Theorem~\ref{t:summation_4_Bn} to prove Theorem~\ref{t:sroots_eq_spsi}.
The proof is similar to that of Theorem~\ref{t:roots_eq_psi}.

\begin{proof}[Proof of Theorem~\ref{t:sroots_eq_spsi}]
	By Theorem~\ref{t:gf_sroot_enumerator_Bn},
	\[
	\sum_{n \ge 0} \sum_{\bmu \vdash n} 
	\sroots_k^{B_n}(\bmu) \,
	\frac{\bt^{\bc(\bmu)}}{|Z_\bmu|}
	= \exp \left(
	\sum_{j,\theta}
	\sum_{\substack{h|k \\ \gcd(h,j) = 1}} 
	\sroots_h^{\ZZ_2}(\theta) \,
	\frac{t_{j,\theta}^{j k/h}}{2 j k/h}         \right) .
	\]
	On the other hand, by Theorem~\ref{t:summation_4_Bn},
	\[
	\sum_{n \ge 0}  
	\sum_{\blambda \vdash n} \sum_{\bmu \vdash n} 
	\psi_{B_n}^\blambda(\bmu) \,       \frac{\bs^{\bc(\blambda)}\bt^{\bc(\bmu)}}{|Z_\bmu|} 
	= \exp \left( 
	\sum_{i,\ve} \sum_{j,\theta} \sum_{e | \gcd(i,j)}  
	K_{\ve,\theta}(e) \,
	\frac{(s_{i,\ve} t_{j,\theta})^{ij/e}}{2ij/e} 
	\right) ,
	\]
	where
	\[
	K_{\ve,\theta}(e)
	:= \ve \theta \cdot \mu(2e)
	+ \frac{(1+\ve)(1+\theta)}{2} \cdot \mu(e).
	\] 
	Letting
	\[
	s_{i,\ve} :=
	\begin{cases}
		1, &\text{if } i|k \text{ and } \ve^{k/i} = -1; \\
		0, &\text{otherwise,}
	\end{cases}
	\]
	we obtain, by Definition~\ref{def:psi_vdash_sk},
	\begin{align*}
		\sum_{n \ge 0} \sum_{\bmu \vdash n} 
		\spsi_k^{B_n}(\bmu) \,
		\frac{\bt^{\bc(\bmu)}}{|Z_\bmu|}
		&= 
		\sum_{n \ge 0} \sum_{\bmu \vdash n} 
		\sum_{\blambda \vdash_\sk n} 
		\psi_{B_n}^\blambda(\bmu) \, 
		\frac{\bt^{\bc(\bmu)}}{|Z_\bmu|} \\
		&= 
		\exp \left( 
		\sum_{j,\theta} 
		\sum_{\substack{i \\ i|k}} 
		\sum_{\substack{\ve \\ \ve^{k/i} = -1}} 
		\sum_{\substack{e \\ e | \gcd(i,j)}}  
		K_{\ve,\theta}(e) \,
		\frac{t_{j,^\theta}^{ij/e}}{2ij/e} 
		\right) .
	\end{align*}
	Fixing the integer $k$, it follows that the claim 
	\[
	\sroots_k^{B_n} = \spsi_k^{B_n}
	\qquad (\forall n \ge 0)
	\]
	is equivalent to the claim that, for any $j \ge 1$ and $\theta \in \ZZ_2$,
	\begin{equation}\label{eq:sroots_hLc_1_Bn}
		\sum_{\substack{h \\ h|k \\ \gcd(h,j) = 1}} 
		\sroots_h^{\ZZ_2}(\theta) \, 
		\frac{t_{j,\theta}^{j k/h}}{2 j k/h}         = \sum_{\substack{i \\ i|k}} 
		\sum_{\substack{\ve \\ \ve^{k/i} = -1}} 
		\sum_{\substack{e \\ e | \gcd(i,j)}}  
		K_{\ve,\theta}(e) \,
		\frac{t_{j,\theta}^{ij/e}}{2ij/e} .
	\end{equation}
	Letting $d := i/e$ and then $h := k/d$ on the RHS of~\eqref{eq:sroots_hLc_1_Bn} yields
	\begin{align*}
		\text{RHS}
		&= 
		\sum_{\substack{i \\ i|k}} 
		\sum_{\substack{e \\ e|i \\ e|j}}  
		\sum_{\substack{\ve \\ \ve^{k/i} = -1}}
		K_{\ve,\theta}(e) \,
		\frac{t_{j,\theta}^{ij/e}}{2ij/e} \\
		&= 
		\sum_{\substack{d \\ d|k}} 
		\sum_{\substack{e \\ e|(k/d) \\ e|j}}  
		\sum_{\substack{\ve \\ \ve^{k/(de)} = -1}} 
		K_{\ve,\theta}(e) \,
		\frac{t_{j,\theta}^{j d}}{2j d} \\
		&= 
		\sum_{\substack{h \\ h|k}} 
		\sum_{\substack{e \\ e|h \\ e|j}}  
		\sum_{\substack{\ve \\ \ve^{h/e} = -1}} 
		K_{\ve,\theta}(e) \,
		\frac{t_{j,\theta}^{j k/h}}{2j k/h} .
	\end{align*}
	Comparing powers of $t_{j,\theta}$ in~\eqref{eq:sroots_hLc_1_Bn}, it follows that we need to prove that, for any $j \ge 1$, $\theta \in \ZZ_2$ and a divisor $h$ of $k$,
	\begin{equation}\label{eq:sroots_hLc_2_Bn}
		\sum_{\substack{e \\ e|\gcd(h,j)}} 
		\sum_{\substack{\ve \\ \ve^{h/e} = -1}} 
		K_{\ve,\theta}(e) 
		= \begin{cases}
			\sroots_h^{\ZZ_2}(\theta), &\text{if } \gcd(h,j) = 1; \\
			0, &\text{otherwise.}
		\end{cases}
	\end{equation}
	Indeed, by 
	the explicit formula for $K_{\ve,\theta}(e)$,
	\[
	\sum_{\substack{\ve \\ \ve^{h/e} = -1}} 
	K_{\ve,\theta}(e) 
	= \begin{cases}
		-\mu(2e), &\text{if } \theta = +1 \text{ and } h/e \text{ is odd;} \\
		\mu(2e), &\text{if } \theta = -1 \text{ and } h/e \text{ is odd;} \\
		0, &\text{if } h/e \text{ is even.}
	\end{cases}
	\]
	Since $\mu(2e) = 0$ for even $e$,
	\begin{align*}
		\sum_{\substack{e \\ e|\gcd(h,j) \\ h/e \text{ odd}}} \mu(2e)
		&= -\sum_{\substack{e \text{ odd} \\ e|\gcd(h,j) \\ h/e \text{ odd}}} \mu(e)
		= \begin{cases}
			-\sum_{e|\gcd(h,j)} \mu(e), &\text{if } h \text{ is odd;} \\
			0, &\text{otherwise} 
		\end{cases} \\
		&= \begin{cases}
			-1, &\text{if } \gcd(h,j) = 1 \text{ and } h \text{ is odd;} \\
			0, &\text{otherwise.} 
		\end{cases}
	\end{align*}
	It follows that
	\[
	\sum_{\substack{e \\ e|\gcd(h,j)}} 
	\sum_{\substack{\ve \\ \ve^{h/e} = -1}} 
	K_{\ve,\theta}(e) 
	= \begin{cases}
		1, &\text{if } \theta = +1,\, \gcd(h,j) = 1 \text{ and } h \text{ is odd;} \\
		-1, &\text{if } \theta = -1,\, \gcd(h,j) = 1 \text{ and } h \text{ is odd;} \\
		0, &\text{otherwise.}
	\end{cases}
	\]
	If $\gcd(h,j) \ne 1$ then the RHS here is zero, and if $\gcd(h,j) = 1$ then the RHS is $\sroots_h^{\ZZ_2}(\theta)$.
	This proves~\eqref{eq:sroots_hLc_2_Bn} and completes the proof of the theorem.
\end{proof}

\subsection{Proofs of Theorems~\ref{t:Dn-higher} and~\ref{t:Dn-proper}}
\label{subsec:D_proper}

\begin{proof}[Proof of Theorem~\ref{t:Dn-higher}]
	Fix an integer $k$, a positive integer $n$ and an element $y \in B_n$.
	By the definitions of the root enumerator $\roots_k^{B_n}$, the signed root enumerator $\sroots_k^{B_n}$ (see Definition~\ref{def:sroots_Bn})
	and the character $\chi$,
	\[
	\roots_k^{B_n}(y) + \sroots_k^{B_n}(y)
	= \sum_{\substack{x \in B_n \\ x^k = y}} 1
	+ \sum_{\substack{x \in B_n \\ x^k = y}} \chi(x)
	= \sum_{\substack{x \in D_n \\ x^k = y}} 2
	= \begin{cases}
		2 \roots_k^{D_n}(y), &\text{if } y \in D_n; \\
		0, &\text{if } y \in B_n \setminus D_n .
	\end{cases}
	\]
	By Theorems~\ref{t:roots_eq_psi} and~\ref{t:sroots_eq_spsi}, it follows that
	\[
	\psi_k^{B_n}(y) + \spsi_k^{B_n}(y)
	= \begin{cases}
		2 \roots_k^{D_n}(y), &\text{if } y \in D_n; \\
		0, &\text{if } y \in B_n \setminus D_n .
	\end{cases}
	\]
	By Definitions~\ref{def:psi_vdash_k} and~\ref{def:psi_vdash_sk}, this implies Theorem~\ref{t:Dn-higher}.
\end{proof}

\begin{proof}[Proof of Theorem~\ref{t:Dn-proper}]
	Consider the expansion of the class function $\roots_k^{D_n}$ as a linear combination of irreducible $D_n$-characters. 
	By Theorem~\ref{t:Dn-higher}, 
	$2 \roots_k^{D_n}$ is a restriction (to $D_n$) of a proper character of $B_n$, and is therefore a proper character of $D_n$.
	In particular, all the coefficients in this expansion are nonnegative. 
	By the result of Frobenius cited at the beginning of Section~\ref{sec:introduction},
	all the coefficients in the expansion of $\roots_k^{D_n}$ are integers.
	This completes the proof. 
\end{proof}

\subsection{Higher Lie Characters of type $D$}
\label{subsec:HLC_Dn}

	%
	%
	%




In this subsection we use Theorem~\ref{t:Dn-higher} to prove Theorem~\ref{t:Dn-higher_odd}. 

\begin{definition}\label{def:psi_D} 
	Let $C$ be a conjugacy class in $D_n$, and let $x \in C$. 
	Consider the linear character $\omega^x = \omega_{B_n}^x$ from Definition~\ref{def:HLC_Bn}, which is defined on $Z_{B_n}(x)$, the centralizer of $x$ in $B_n$.
	Define
	\[
	\omega_{D_n}^x 
	:= \omega^x\downarrow_{Z_{D_n}(x)}^{Z_{B_n}(x)}
	\]
	and
	\[
	\psi_{D_n}^x
	:= \omega_{D_n}^x\uparrow_{Z_{D_n}(x)}^{D_n} .
	\]
	As in the case of $B_n$, the character $\psi_{D_n}^x$ is independent of the choice of $x \in C$, and will therefore be denoted $\psi_{D_n}^C$.
\end{definition}

	%




\begin{lemma}\label{t:D_conj_and_cent}
	For $x \in D_n$, the following conditions are equivalent.
	\begin{enumerate}
		
		\item[(a)]
		The conjugacy class of $x$ in $D_n$ is equal to its conjugacy class in $B_n$.
		
		\item[(b)]
		The index $[Z_{B_n}(x):Z_{D_n}(x)] = 2$.
		
		\item[(c)]
		The centralizer $Z_{B_n}(x) \not\subseteq D_n$.
		
		\item[(d)]
		The element $x$ has a cycle of class $-1 \in \ZZ_2$ or a cycle of odd length.
		
	\end{enumerate}
\end{lemma}

\begin{proof} 
	Conditions (a), (b) and (c) are equivalent since $[B_n:D_n] = 2$.
	Conditions (c) and (d) are equivalent by Lemma~\ref{t:centralizer_in_Bn}, describing the exact structure of the centralizer $Z_{B_n}(x)$.  
\end{proof}

\begin{lemma}\label{t:HLC_D}
	If the equivalent conditions in Lemma~\ref{t:D_conj_and_cent} are satisfied, and $C$ denotes the conjugacy class in condition (a), then
	\[
	\psi_{D_n}^C 
	= \psi_{B_n}^C\downarrow_{D_n}^{B_n}. 
	\]
\end{lemma}    

\begin{proof} 
	If $Z_{B_n}(x) \not\subseteq D_n$ then $B_n = D_n Z_{B_n}(x)$, and of course $Z_{D_n}(x) = D_n \cap Z_{B_n}(x)$. Hence, by \cite[Problem~(5.2)]{Isaacs}, a simple case of Mackey's Theorem, 
	\[
	\psi_{D_n}^C 
	= \omega^x_{D_n}\uparrow_{Z_{D_n}(x)}^{D_n}
	= \omega^x_{B_n}\downarrow_{D_n \cap Z_{B_n}(x)}^{Z_{B_n}(x)} \uparrow^{D_n}
	= \omega^x_{B_n}\uparrow_{Z_{B_n}(x)}^{D_n Z_{B_n}(x)} \downarrow_{D_n}
	= \psi_{B_n}^C \downarrow_{D_n}^{B_n} . 
	\qedhere
	\]
\end{proof}

We also need the following general property of root enumerators.

\begin{lemma}\label{t:index2_kodd}
	If $H$ is an index $2$ subgroup of a finite group $G$ then, for every odd $k$, $\roots^H_k = \roots^G_k \downarrow_H^G$.
\end{lemma}

\begin{proof}
	As a subgroup of index $2$ in $G$, $H \triangleleft G$. It is therefore the kernel of a linear $G$-character, namely $\sign_H: G \to \{1,-1\}$ defined by
	\[
	\sign_H(x) 
	:= \begin{cases}
		1, &\text{if } x \in H; \\
		-1, &\text{if } x \in G \setminus H.
	\end{cases}
	\]
	Since $k$ is odd, it follows that, for any $x, y \in G$: 
	\[
	x^k = y \,\Longrightarrow\, 
	\sign_H(x) = \sign_H(y).
	\] 
	Thus, all $k$-th roots (in $G$) of an element of $H$ belong to $H$, namely $\roots^H_k = \roots^G_k \downarrow_H^G$.
\end{proof}

	
\begin{proof}[Proof of Theorem~\ref{t:Dn-higher_odd}]
	Assume first that $k$ is odd.
	By Lemma~\ref{t:index2_kodd} and Theorem~\ref{t:Bn-higher},
	\[
	\roots^{D_n}_k
	= \roots^{B_n}_k\downarrow^{B_n}_{D_n}
	= \sum_{C \in \Conj(B_n) \,:\, C^k = \{1\}} \psi_{B_n}^C\downarrow^{B_n}_{D_n} .
	\]
	If $C \in \Conj(B_n)$ satisfies $C^k = \{1\}$ then, by the proof of Lemma~\ref{t:index2_kodd}, necessarily $C \subseteq D_n$.
	Each element of $C$ has a cycle of odd length (in fact, all its cycles have odd lengths, since $k$ is odd) and therefore, by Lemma~\ref{t:D_conj_and_cent}, $C$ is equal to the $D_n$-conjugacy class of each of its elements, namely $C \in \Conj(D_n)$.
	Conversely, if $C \in \Conj(D_n)$ satisfies $C^k = \{1\}$ then $C \in \Conj(B_n)$.
	Therefore, by Lemma~\ref{t:HLC_D},
	\[
	\roots^{D_n}_k
	= \sum_{C \in \Conj(D_n) \,:\, C^k = \{1\}} 
	\psi_{B_n}^C \downarrow^{B_n}_{D_n} 
	= \sum_{C \in \Conj(D_n) \,:\, C^k = \{1\}} \psi_{D_n}^C,
	\]
	as claimed.
	
	Assume now that $k$ is even and $n$ is odd.
	Then $w_0 \not\in D_n$, so $C^k = \{w_0\}$ is impossible. Hence, by Theorem~\ref{t:Dn-higher},
	\[
	2\roots^{D_n}_k
	= \sum_{C \in \Conj(B_n) \,:\, C^k = \{1\}} \psi_{B_n}^C \downarrow_{D_n}^{B_n} .  
	\]
	Since $D_n \triangleleft B_n$, each conjugacy class $C \in \Conj(B_n)$ is contained in either $D_n$ or $B_n \setminus D_n$.
	Since $n$ is odd, each element of $B_n$ has a cycle of odd length, and therefore, by Lemma~\ref{t:D_conj_and_cent}, if $C \in \Conj(B_n)$ is contained in $D_n$ then $C \in \Conj(D_n)$, and conversely: if $C \in \Conj(D_n)$ then $C \in \Conj(B_n)$. 
	Moreover, if $C \in \Conj(B_n)$ then also $w_0 C \in \Conj(B_n)$, and exactly one of $\{C,\,w_0 C\}$ is contained in $D_n$.
	Since $k$ is even, $C^k = 1$ if and only if $(w_0 C)^k = 1$.
	Hence
	\[
	2\roots^{D_n}_k
	= \sum_{C \in \Conj(D_n) \,:\, C^k = \{1\}} \psi_{B_n}^C \downarrow_{D_n}^{B_n}
	+ \sum_{C \in \Conj(D_n) \,:\, C^k = \{1\}} \psi_{B_n}^{w_0 C} \downarrow_{D_n}^{B_n}.
	\]
	By Definition~\ref{def:HLC_Bn}, for $x\in B_n$,  $\omega^{w_0x} = \omega^x \cdot (\chi\downarrow_{Z_{B_n}(x)}^{B_n})$, as linear characters of $Z_{B_n}(w_0x)=Z_{B_n}(x)$. 
	Thus, by \cite[Problem~(5.3)]{Isaacs}, $\psi^{w_0 C}_{B_n} = \psi^C_{B_n} \chi$ for $C \in \Conj(B_n)$. 
	Since $\chi\downarrow_{D_n} = 1$, we deduce that
	\[
	2\roots^{D_n}_k
	= \sum_{C \in \Conj(D_n) \,:\, C^k = \{1\}} (\psi^C_{B_n} (1 + \chi)) \downarrow_{D_n}^{B_n} 
	= \sum_{C \in \Conj(D_n) \,:\, C^k = \{1\}} 2 \psi^{C}_{B_n} \downarrow_{D_n}^{B_n},
	\]
	and using Lemma~\ref{t:HLC_D} we conclude that
	\[
	\roots^{D_n}_k
	= \sum_{C \in \Conj(D_n) \,:\, C^k = \{1\}} \psi^C_{D_n}.
	\qedhere
	\]
	
	
\end{proof}

\section{Final remarks and open problems}
\label{sec:final}

\subsection{Gelfand models}
\label{subsec:Gelfand}

A representation $\rho$ of a finite group $G$ 
is a {\em Gelfand model} for $G$ if the multiplicity in $\rho$ of every irreducible representation of $G$ is 1. 
The Inglis-Richardson-Saxl construction of a  Gelfand model for $S_n$~\cite{Inglis}  
and Baddeley's construction of a Gelfand model for $B_n$~\cite{Baddeley}  
may be reformulated as follows.

\begin{theorem}\label{t:ABn_Gelfand}
	The Gelfand model of the Weyl group of type $A_{n-1}$ (or $B_n$) is isomorphic to the multiplicity-free sum of the higher Lie representations 
	of $A_{n-1}$ (respectively, $B_n$) on all conjugacy classes of involutions in $A_{n-1}$ (respectively, $B_n$).
\end{theorem} 

As noted in~\cite{Baddeley}, 
a similar construction of a Gelfand model for $D_n$ as a sum of higher Lie characters exists when $n$ is odd but not when $n$ is even,  

\begin{definition}
	A representation $\rho$ of a finite group $G$ is a {\em double Gelfand model} for $G$ if the multiplicity in $\rho$ of every irreducible $G$-representation is 2. 
\end{definition}

Theorem~\ref{t:Dn-higher} implies the following.

\begin{theorem}\label{t:Dn_Gelfand}
	The sum of the restrictions to $D_n$ of the higher Lie characters of type $B_n$, over all conjugacy classes of involutions or square roots of the longest element in $B_n$, is a double Gelfand model for $D_n$.
\end{theorem} 

\begin{proof}
	It follows from the 
	Frobenius-Schur indicator theorem  that for any real represented finite group $G$, $\roots^{G}_2$ is a multiplicity-free sum of all irreducible $G$-characters, see~\cite[\S 4]{Isaacs} and~\cite[Ex. 7.69]{ECII}. 
	Since $D_n$ is real represented, Theorem~\ref{t:Dn-higher} completes the proof. 
\end{proof}

\subsection{Index 2 subgroups}
\label{subsec:index2} 

We proved (Theorem~\ref{t:Dn-proper}) that the root enumerators in the classical Weyl group of type $D_n$ are proper characters, by considering $D_n$ as an index 2 subgroup of $B_n$. 
It is natural to consider other index 2 subgroups of the classical Weyl groups. 

\begin{problem}\label{prob:index2}
	Let $H$ be a subgroup of index $2$ of a classical Weyl group of type $A$, $B$ or $D$. For which integers $k$ is the root enumerator $\roots^H_k$ a proper character?
\end{problem}

The answer is positive for $k=2$ and any odd $k$, by Propositions~\ref{t:Coxeter_k_2} and~\ref{t:Weyl_k_odd} below. 
It is tempting to conjecture that this holds for all values of $k$, but this is not the case. 
The counterexample of smallest rank that we know is the alternating subgroup $A(B_{10})$, for which the root enumerators for $k=10$ and $k=70$ are not proper characters. 

	
	

\begin{proposition}\label{t:Coxeter_k_2}
	If $H$ is an index $2$ subgroup of a finite Coxeter 
	group, then the square root enumerator $\roots_2^H$ is a proper $H$-character.    
\end{proposition}

\begin{proof}
	Every finite Coxeter group is {\em totally orthogonal}, i.e., all its irreducible representations are real; 
	see, e.g., \cite{LSV}. 
	By~\cite[Proposition 3.6]{WG}, if $H$ is an index $2$ subgroup of a totally orthogonal group then any real character of $H$ comes from a real representation.
	It follows, by the Frobenius-Schur indicator theorem (see~\cite[\S 4]{Isaacs}), that $\roots_2^H$ is a multiplicity-free sum of certain (not necessarily all) irreducible $H$-characters, and is therefore a proper character.
\end{proof}

\begin{proposition}\label{t:Weyl_k_odd}
If $H$ is an index $2$ subgroup of a Weyl group of type $A$, $B$ or $D$, 
then $\roots^H_k$ is proper for every odd  $k$.  
\end{proposition}

\begin{proof}
By Corollary~\ref{t:Weyl-proper}, 
all root enumerators in the classical Weyl groups of types $A$, $B$ and $D$ are proper.
By Lemma~\ref{t:index2_kodd}, all $k$-th root enumerators with odd $k$ in their index $2$ subgroups are proper, as restrictions of proper characters. 
\end{proof}    

%

%


\begin{observation}\label{t:ABn}
If $n$ is odd then 
$A(B_n)\cong D_n$, 
hence $\roots^{A(B_n)}_k$ is proper for any $k$. 
\end{observation}

\begin{proof}
For every $n$, 
$w_0 := [-1,\ldots,-n] \in B_n$ is a central involution.
If $n$ is odd then $w_0 \not\in D_n$ as well as $w_0 \not\in A(B_n)$. Hence
\[
B_n 
\cong D_n \times \ZZ_2 
\cong A(B_n) \times \ZZ_2. 
\]
Thus $D_n \cong B_n/\ZZ_2 \cong A(B_n)$. 
Theorem~\ref{t:Dn-proper} completes the proof. 
\end{proof}


\medskip

Hereby we list explicit generating functions for $\roots_k^H$, where $H$ is $\ZZ_2 \wr A(S_n)$, $A(B_n)$ or $A(D_n)$. 
Recall, from Definition~\ref{def:sign_char_Bn}, that the linear character $\chi$ is defined on $B_n = \ZZ_2 \wr S_n$ as follows:
$\chi$ is the identity map (i.e., the sign character) on $\ZZ_2 = \{+1,-1\}$, and is trivial on the wreathing group $S_n$.
Define now a linear character $\chi'$ on $B_n$ as follows:
$\chi'$ is trivial on $\ZZ_2$, and is the sign character on the wreathing group $S_n$.
For $n \ge 2$, the group $B_n$ has three nontrivial linear characters: $\chi$, $\chi'$, and $\chi \chi'$.
They define the three index $2$ subgroups of $B_n$:
$D_n = \ker(\chi)$,
$\ZZ_2 \wr A(S_n) = \ker(\chi')$, and
$A(B_n) = \ker(\chi \chi')$.

Recall also, from Definition~\ref{def:sroots_Bn}, that
\[
\sroots_k^{B_n}(y) 
:= \sum_{\substack{x \in B_n \\ x^k = y}} \chi(x)
\qquad (y \in B_n) .
\]

Define analogous class functions for $\chi'$ and $\chi \chi'$.

\begin{definition}\label{def:sroots_Bn_index_2}
	Let $k$ be an integer and $n$ a nonnegative integer. Define
	\[
	\roots_k^{\prime B_n}(y) 
	:= \sum_{\substack{x \in B_n \\ x^k = y}} \chi'(x)
	\qquad (y \in B_n) 
	\]
	and
	\[
	\sroots_k^{\prime B_n}(y) 
	:= \sum_{\substack{x \in B_n \\ x^k = y}} \chi \chi'(x)
	\qquad (y \in B_n) .
	\]
\end{definition}

The corresponding analogue of Theorem~\ref{t:gf_sroot_enumerator_Bn}, providing explicit formulas for the corresponding generating functions, is the following.

\begin{theorem}\label{t:gf_sroot_enumerator_Bn_index_2}
	For any integer $k$, 
	\[
	\sum_{n \ge 0} 
	\sum_{\bmu \vdash n} 
	\roots_k^{\prime B_n}(\bmu) \,
	\frac{\bt^{\bc(\bmu)}}{|Z_\bmu|}
	= \exp \left(
	\sum_{j,\theta}
	\sum_{\substack{h|k \\ \gcd(h,j) = 1}} 
	\roots_h^{\ZZ_2}(\theta) \,
	\frac{-(-t_{j,\theta})^{j k/h}}{2 j k/h} 
	\right) 
	\]
	and
	\[
	\sum_{n \ge 0} 
	\sum_{\bmu \vdash n} 
	\sroots_k^{\prime B_n}(\bmu) \,
	\frac{\bt^{\bc(\bmu)}}{|Z_\bmu|}
	= \exp \left(
	\sum_{j,\theta}
	\sum_{\substack{h|k \\ \gcd(h,j) = 1}} 
	\sroots_h^{\ZZ_2}(\theta) \,
	\frac{-(-t_{j,\theta})^{j k/h}}{2 j k/h} 
	\right) ,
	\]
	where 
	$\roots_h^{\ZZ_2}(\theta) = \sum_{\ve \in \ZZ_2,\, \ve^h = \theta} 1$ 
	and 
	$\sroots_h^{\ZZ_2}(\theta) = \sum_{ve \in \ZZ_2,\, \ve^h = \theta} \chi(\ve)$ 
	$(h \in \ZZ,\, \theta \in \ZZ_2)$.
\end{theorem}

\begin{proof}
	The proof of these formulas is very similar to the proof of Theorem~\ref{t:gf_sroot_enumerator_Bn}, carried out in Subsection~\ref{subsec:gf_signed_roots}.
	We record only a few milestones.
	The analogues of Formula~\eqref{eq:chi_1} (in the proof of Lemma~\ref{t:sroot_enumerator_similar_cycles_Bn}) are
	\[
	\chi'(x') = \chi'(x) 
	= \prod_{i=1}^{q} 
	\prod_{\ve \in R_{k/d_i,\theta}}
	(-1)^{(\ell_i - 1) m_{i,\ve}} 
	\]
	and
	\[
	\chi\chi'(x') = \chi\chi'(x) 
	= \prod_{i=1}^{q} 
	\prod_{\ve \in R_{k/d_i,\theta}}
	(-1)^{(\ell_i - 1) m_{i,\ve}} 
	\ve^{m_{i,\ve}} ,
	\]
	respectively.
	Consequently, the analogues of Formula~\eqref{eq:chi_2} (in the same proof) are
	\[
	b! \left( 2j \right)^b \cdot 
	\prod_{i=1}^{q} \frac{1}{(2 \ell_i)^{m_i}} 
	\prod_{\ve \in R_{k/d_i,\theta}} \frac{(-(-r_{\ell_i,\ve})^{\ell_i})^{m_{i,\ve}}}{m_{i,\ve}!} 
	\]
	and
	\[
	b! \left( 2j \right)^b \cdot 
	\prod_{i=1}^{q} \frac{1}{(2 \ell_i)^{m_i}} 
	\prod_{\ve \in R_{k/d_i,\theta}} \frac{(-\ve (-r_{\ell_i,\ve})^{\ell_i})^{m_{i,\ve}}}{m_{i,\ve}!} ,
	\]  
	respectively.
\end{proof}

\begin{remark}
	The corresponding expression for the sign character of the symmetric group $S_n$ was obtained by Chernoff~\cite{Chernoff} and by Glebsky, Lic\'on and Rivera~\cite{GLR}.
	In our notation, including the class function $\roots_k^{\prime S_n}(y) := \sum_{x \in S_n,\, x^k = y} \sign(x)$ (for $y \in S_n$), namely $\roots_k^{\prime S_n}(\mu)$ (for a cycle type $\mu \vdash n$),
	their result is:
	\[
	\sum_{n \ge 0} 
	\sum_{\mu \vdash n} 
	\roots_k^{\prime S_n}(\mu) \,
	\frac{\bt^{\bc(\mu)}}{|Z_\mu|}
	= \exp \left(
	\sum_{j \ge 1}
	\sum_{\substack{h|k \\ \gcd(h,j) = 1}} 
	\frac{-(-t_j)^{j k/h}}{j k/h}
	\right) .
	\]
\end{remark}

The following formulas are natural analogues of the corresponding formula for $\roots_k^{D_n}$.

\begin{observation}
	Let $k$ be an integer and $n$ a nonnegative integer. Then, for any $y \in B_n$:
	\[
	\roots_k^{B_n}(y) + \roots_k^{\prime B_n}(y) =
	\begin{cases}
		2 \roots_k^{\ZZ_2 \wr A(S_n)}(y), &\text{if } y \in \ZZ_2 \wr A(S_n); \\
		0, &\text{otherwise,}
	\end{cases}
	\]
	\[
	\roots_k^{B_n}(y) + \sroots_k^{\prime B_n}(y) =
	\begin{cases}
		2 \roots_k^{A(B_n)}(y), &\text{if } y \in A(B_n); \\
		0, &\text{otherwise,}
	\end{cases}
	\]
	\[
	\roots_k^{B_n}(y) + \sroots_k^{B_n}(y) + \roots_k^{\prime B_n}(y) + \sroots_k^{\prime B_n}(y) =
	\begin{cases}
		4 \roots_k^{A(D_n)}(y), &\text{if } y \in A(D_n); \\
		0, &\text{otherwise.}
	\end{cases}
	\]
\end{observation}

\subsection{Quasisymmetric functions}
\label{subsec:QSF}

In their seminal paper~\cite{GR},  Gessel and Reutenauer show that for every partition $\lambda\vdash n$, 
the Frobenius image of $\psi_{S_n}^\lambda$, the higher Lie character of type $A$, is equal to the sum of the fundamental quasisymmetric functions corresponding to the descent sets of all the permutations in the conjugacy class of cycle type $\lambda$. 


\begin{conjecture}\label{conj:QSF}
	For every bipartition $\blambda = (\lambda^+,\lambda^-) \vdash n$, 
	the Frobenius characteristic image of the type $B$ higher Lie character $\psi_{B_n}^\blambda$ is equal to the sum of  Poirier's fundamental quasisymmetric functions corresponding to the signed descent sets of all the 
	elements in the conjugacy class of signed cycle type $\blambda$.    
\end{conjecture}


\begin{remark}
	It is well-known that, for every $k$, the Frobenius image of the $k$-th root enumerator in $S_n$ is equal to the sum of the fundamental quasisymmetric functions corresponding to the descent sets of all the $k$-th roots of the identity; see, e.g., \cite[Theorem 7.7]{AAER}.
	An analogue for $k$-th root enumerators in $B_n$, when $k$ is odd, was proved in
	~\cite[Theorem 7.8]{AAER},   
	and left open for even $k$. 
	Conjecture~\ref{conj:QSF}, combined with Theorem~\ref{t:Bn-higher}, will imply this type $B$ analogue for all $k$. 
\end{remark}

	
	
	
		

\end{document}